\documentclass[11pt]{amsart}
\usepackage{amsfonts,amsmath,amssymb,amsthm}
\usepackage{color,easybmat,enumerate,graphicx,lmodern,perpage,textcomp}

\usepackage[margin=1.9cm, a4paper]{geometry}

\theoremstyle{plain}
\newtheorem{theorem}{Theorem}[section]
\newtheorem*{theoremnn}{Theorem A.1}
\newtheorem{conjecture}[theorem]{Conjecture}
\newtheorem{corollary}[theorem]{Corollary}

\newtheorem{lemma}[theorem]{Lemma}
\newtheorem{proposition}[theorem]{Proposition}

\theoremstyle{remark}

\DeclareMathOperator*{\supp}{supp}
\numberwithin{equation}{section}

\MakePerPage{footnote} \allowdisplaybreaks 

\newcommand{\recip}[1]{\frac{1}{#1}}
\newcommand{\Dl}{\mathcal{D}_\lambda^+}
\newcommand{\e}{\epsilon}

\newcommand{\norm}[1]{\|{#1}\|}
\newcommand\pO{\partial\Omega}
\newcommand\ra{\rightarrow}
\newcommand\R{\mathbb R}
\newcommand{\re}{\mathbb R}
\newcommand{\Qs}{Q^0_{\lambda,\Gamma_1,\Gamma_2}}
\newcommand{\Qa}{Q^1_{\lambda, \Gamma_1,\Gamma_2}}

\newcommand{\Sl}{\mathcal{S}_{\lambda}^+}

\renewcommand{\Re}{\mathop{\rm Re}\nolimits}
\renewcommand{\Im}{\mathop{\rm Im}\nolimits}

\title[Layer potentials and operators]
{Sharp norm estimates of layer potentials and operators at high frequency}

\author{Xiaolong Han}
\address{Department of Mathematics, Australian National University, Canberra, ACT 0200, Australia}
\email{Xiaolong.Han@anu.edu.au}

\author{Melissa Tacy\\ \\{\smaller with an appendix by} Jeffrey Galkowski}
\address{School of Mathematics, University of Adelaide, Adelaide, SA 5005, Australia}
\email{melissa.tacy@adelaide.edu.au}

\subjclass[2010]{31B10, 31B25, 31B35}

\keywords{layer potentials, layer operators, high frequency, eigenfunctions, quasimodes, restriction estimates, boundedness, sharpness}

\begin{document}

\begin{abstract}
In this paper, we investigate single and double layer potentials mapping boundary data to interior functions of a domain at high frequency $\lambda^2\to\infty$. For single layer potentials, we find that the $L^{2}(\partial\Omega)\to{}L^{2}(\Omega)$ norms decay in $\lambda$. The rate of decay depends on the curvature of $\partial\Omega$: The norm is $\lambda^{-3/4}$ in general domains and $\lambda^{-5/6}$ if the boundary $\pO$ is curved. The double layer potential, however, displays uniform $L^{2}(\partial\Omega)\to{}L^{2}(\Omega)$ bounds independent of curvature. By various examples, we show that all our estimates on layer potentials are sharp. 

The appendix by Galkowski gives bounds $L^{2}(\partial\Omega)\to{}L^{2}(\partial\Omega)$ for the single and double layer operators at high frequency that are sharp modulo $\log \lambda$. In this case, both the single and double layer operator bounds depend upon the curvature of the boundary.
\end{abstract}

\maketitle 
\section{Introduction}\label{introduction}

Denote $\Delta=\sum_{i=1}^n\partial_i^2$ as the Laplacian operator in $\R^n$. Given a piecewise smooth and bounded domain $\Omega\subset\R^n$, Green's formula yields that the solution to the Helmholtz equation in $\Omega$
$$-\Delta u=\lambda^2 u$$
has the form
\begin{equation}\label{eq:Green}
u(x)=\int_{\partial\Omega}K_\lambda(x-y)\partial_{\nu_y} u(y)d\sigma_y-\int_{\partial\Omega}\partial_{\nu_y} K_\lambda(x-y)u(y)d\sigma_y,
\end{equation}
where $\partial_{\nu_y}$ is the outward normal derivative at $y\in\pO$, $d\sigma$ is the surface measure on $\pO$, and $K_\lambda$ is a fundamental solution to $(-\Delta-\lambda^2)$, that is
\begin{equation}\label{eq:fund}
(-\Delta-\lambda^2)K_\lambda(x)=\delta(x).
\end{equation}
In fact, we can write $K_\lambda$ explicitly as
$$K^+_{\lambda}(x)=\frac{i}{4}\left(\frac{\lambda}{2\pi|x|}\right)^{\frac{n-2}{2}}H^{(1)}_{\frac{n-2}{2}}(\lambda|x|),$$
where $H^{(1)}_{\frac{n-2}{2}}$ is the Hankel function of the first kind and order $\frac{n-2}{2}$. It is the kernel of the outgoing resolvent $R^+_\lambda=\left[-\Delta-(\lambda+i0)^2\right]^{-1}$.

Let us therefore define the semiclassical single layer potential $S^+_\lambda$ as
$$S_\lambda^+(f)=R^+_\lambda(fd\sigma)=K^+_\lambda\ast(fd\sigma),$$
and the semiclassical double layer potential $D^+_\lambda$ as
$$D_\lambda^+(f)=\partial_\nu K^+_\lambda\ast(fd\sigma).$$
Now \eqref{eq:Green} can be written
\begin{equation}\label{eq:represent}
u=S^+_\lambda(\partial_\nu u)-D_\lambda^+(u|_{\partial\Omega}),
\end{equation}
allowing us to construct interior eigenfunction from boundary data. In particular, 
$$\begin{cases}
u=S_\lambda^+(\partial_\nu u), & \text{ for Dirichlet eigenfunction } u|_{\pO}=0,\\
u=-D_\lambda^+(u|_{\pO}), & \text{ for Neumann eigenfunction } \partial_\nu u=0.
\end{cases}$$

\subsection*{Main results}
In this paper we obtain sharp $L^{2}(\partial\Omega)\to{}L^{2}(\Omega)$ bounds on $S^{+}_{\lambda}$ and $D^{+}_{\lambda}$.
\begin{theorem}[Boundedness and sharpness of semiclassical single layer potentials]\label{thm:SLP}\hfill
\begin{enumerate}[(i).]
\item In a general domain $\Omega$,
$$\|S^+_\lambda(u)\|_{L^2(\Omega)}\le c\lambda^{-\frac34}\|u\|_{L^2(\partial\Omega)},$$
where $c$ depends only on $\Omega$. Furthermore, the exponent $-3/4$ is sharp if the boundary $\pO$ contains a flat piece.
\item If $\partial\Omega$ is curved, that is, the second fundamental form of $\partial\Omega$ is (positive or negative) definite, then
$$\|S^+_\lambda(u)\|_{L^2(\Omega)}\le c\lambda^{-\frac56}\|u\|_{L^2(\partial\Omega)},$$
where $c$ depends only on $\Omega$. Furthermore, the exponent $-5/6$ is sharp if $\Omega$ is an annulus.
\end{enumerate}
\end{theorem}

In contrast to the case of single layer potentials, we obtain a uniform bound for double layer potentials in all domains. The sharp examples for the single layer potential all rely on concentration in the tangential direction (see Section \ref{sharpness}). However the symbol of $D_{\lambda}^{+}$ is zero in such tangential directions.  
\begin{theorem}[Boundedness and sharpness of semiclassical double layer potentials]\label{thm:DLP}
$$\|D^+_\lambda(u)\|_{L^2(\Omega)}\le c\|u\|_{L^2(\partial\Omega)},$$
where $c$ depends only on $\Omega$. Furthermore, the estimate is sharp if $\Omega$ is a disc.
\end{theorem}

The representation of eigenfunctions in \eqref{eq:represent} has applications in a variety of both theoretical and numerical studies. Hassell and Zelditch \cite{HZ} use it to prove the quantum ergodicity of boundary values of eigenfunctions. In particular, they express boundary traces of Dirichlet, Neumann, and Robin eigenfunctions as eigenfunctions of integral operators produced by semiclassical layer potentials. 

In a similar vein Toth and Zelditch \cite{TZ12, TZ13} recently applied these potentials to prove quantum ergodic restriction (QER) theorems on interior hypersurfaces. 

If the boundary $\pO$ is analytic, then there exist analytic continuations of $S^+_\lambda$ and $D^+_\lambda$ in the Grauert tube (a complex neighborhood of the real domain). Such complexification enables the study of zeros of eigenfunctions in the complex region (instead of in the real region), where it has a simpler characterization. For detailed discussion on the nodal intersection estimates, see Toth and Zelditch \cite{TZ09}, El-Hajj and Toth \cite{ET}. 

In star shaped domains Barnett and Hassell \cite{BH} develop a numerical technique for constructing Dirichlet eigenfunctions by solving a related eigenfunction problem on the boundary. They then use \eqref{eq:represent} to reconstruct interior eigenfunctions. Their technique allows them to control error on the boundary and so mapping norms on $S^{+}_{\lambda}$ control error in the interior.

In \cite[Remark 3.2]{BH}, the authors proposed the question of finding a bound on the $\lambda$-dependence of the single layer potential $S^+_\lambda$. In particular they ask is 
\begin{equation}\label{eq:BH}
\norm{S^{+}_{\lambda}}_{L^{2}(\partial\Omega)\to{}L^{2}(\Omega)}\lesssim{}\lambda^{-1}?
\end{equation}
Such a bound would imply that boundary error controls interior error with no loss, which would be optimal for their numerical technique. In fact, some results were achieved previously. Feng and Sheen \cite{FS} showed that the norm is uniformly bounded independent of $\lambda$. Spence \cite{Sp} improved  this to $\lambda^{-1/2}$.

Theorem \ref{thm:SLP} answers Barnett and Hassell's question in the negative for general (and even curved) domains. However, the estimate \eqref{eq:BH} might hold for strictly convex domains; see Conjecture \ref{conj:SLPconvex}.

Weaker estimates than those of Theorem \ref{thm:DLP} were obtained by Feng and Sheen \cite{FS} and Spence \cite{Sp}. Precisely, Feng and Sheen \cite{FS} proved $\|D_\lambda\|_{L^2(\pO)\to L^2(\Omega)}\lesssim\lambda$; while Spence \cite{Sp} improved to $\|D_\lambda\|_{L^2(\pO)\to L^2(\Omega)}\lesssim\lambda^{1/2}$. See also the survey article \cite{C-WGLS} for related results and their applications in numerical computations.

One may compare the high frequency ($\lambda\to{}\infty$) results of Theorems \ref{thm:SLP} and \ref{thm:DLP} with the case $\lambda=0$. There \eqref{eq:Green} reduces to the construction of a harmonic function $u$ in $\Omega$ by its boundary data. $K_0^+=N$ is the fundamental solution of the Laplacian: $-\Delta N(x)=\delta(x)$. The two convolution-type operators $S(f)$ and $D(f)$ mapping $L^2(\pO)$ to $L^2(\Omega)$ with kernels $N$ and $\partial_\nu N$ are the classical single and double layer potentials at zero frequency. The mapping properties of these layer potentials from boundary data to interior functions and related boundary value problems have been studied extensively over the past century. See \cite{FJR, JK, V} and \cite{CK, F} for a detailed discussion of this classical problem. 

In the case of classical single and double layer potentials, it is the regularity of the boundary that determines mapping properties, with the case of smooth boundaries being rather trivial. By contrast, in the high frequency limit $\lambda\to\infty$, the interest is not the boundedness of $S^+_\lambda$ and $D^+_\lambda$ for a particular $\lambda$, but rather the rate of decay of the mapping norms as $\lambda\to\infty$. In this limit, the problem is interesting even for smooth boundaries, as Theorem \ref{thm:SLP} shows, the rate of decay depends on the geometric properties of the boundaries.

The single and double layer operators $\mathcal S_\lambda^+$ and $\mathcal D_\lambda^+$ are restrictions of layer potentials $S_\lambda^+$ and $D_\lambda^+$ to the boundary, that is, $\Dl:f\to D^+_\lambda(f)|_{\pO}$ and $\Sl:f\to S^+_\lambda(f)|_{\pO}$. Galkowski's appendix to this paper provides estimates for both $\Dl$ and $\Sl$. These operators have been studied by other mathematicians and some results were known prior to the above estimates. Please refer to the appendix for a brief discussion. 

\begin{theoremnn}[Boundedness and sharpness of semiclassical layer operators]
$$
\|\Sl\|_{L^2(\partial\Omega)\to L^2(\partial\Omega)}\le
\begin{cases}
C\lambda^{-1/2}\log\lambda&\quad\text{in general domains};\\
C\lambda^{-2/3}\log\lambda&\quad\text{if $\pO$ is curved},
\end{cases}$$
and
$$
\|\Dl\|_{L^2(\partial\Omega)\to L^2(\partial\Omega)}\le
\begin{cases} 
C\lambda^{1/4}\log\lambda&\quad\text{in general domains};\\
C\lambda^{1/6}\log\lambda&\quad\text{if $\pO$ is curved}.
\end{cases}$$
Moreover, these estimates are sharp modulo the $\log\lambda$. 
\end{theoremnn}

\subsection*{Connection with boundary estimates of eigenfunctions}
Because of \eqref{eq:represent}, there is a close relation between semiclassical layer potentials and boundary estimates of eigenfunctions.

\begin{itemize}
\item Dirichlet eigenfunction: $u$ satisfies $u=S^+_\lambda(\partial_\nu u)$. Bardos, Lebeau, and Rauch \cite{BLR} and Hassell and Tao \cite{HTao02, HTao10} proved that
$$\|u\|_{L^2(\Omega)}\approx\lambda^{-1}\|\partial_\nu u\|_{L^2(\pO)},$$
as $u=S^{+}_{\lambda}(\partial_{\nu}u)$ this implies that
$$\norm{S^{+}_{\lambda}}_{L^{2}(\partial\Omega)\to{}L^{2}(\Omega)}\ge c\lambda^{-1}.$$
Therefore the sharp examples for Theorem \ref{thm:SLP} that we produce in Section \ref{sharpness} are far from being normal derivatives of a Dirichlet eigenfunctions.

\item Neumann eigenfunction: $u$ satisfies $u=D^+_\lambda(u|_{\pO})$. Hence, as a corollary of Theorem \ref{thm:DLP}, we have
\begin{corollary}[Boundary estimate of Neumann eigenfunctions]
Let $u$ be a Neumann eigenfunction of $\Delta$ in $\Omega$ with eigenvalue $\lambda^2$. Then
$$\|u\|_{L^2(\Omega)}\le c\|u\|_{L^2(\pO)},$$
where $C$ is independent of $\lambda$.
\end{corollary}

Tataru \cite{Ta} proved that
\begin{equation}\label{eq:Tataru}
\|u\|_{L^2(\pO)}\lesssim\lambda^\frac13\|u\|_{L^2(\Omega)}.
\end{equation}
Putting the above results together gives both lower and upper bounds of the boundary estimates of Neumann eigenfunctions:
$$\|u\|_{L^2(\Omega)}\lesssim\|u\|_{L^2(\pO)}\lesssim\lambda^\frac13\|u\|_{L^2(\Omega)}.$$
Furthermore, this suggests that the sharp examples for Tataru's estimate are far from saturating the double layer potential estimates. We will discuss the sharpness of above inequalities in Section \ref{sharpness}, and construct examples for which saturation is achieved.
\end{itemize}

\subsection*{Connection with interior hypersurface restriction estimates of eigenfunctions}

One can similarly define the incoming resolvent 
$$R^-_\lambda=\left[-\Delta-(\lambda-i0)^2\right]^{-1}$$ 
and $K_\lambda^-$ as its kernel. Then the above theorem is also valid with the same norm for $S^-_\lambda$. We will drop the $\pm$ sign in the $L^2$ estimates without causing any confusion. If we write $dE_\lambda$ as the spectral measure operator $\delta(-\Delta-\lambda^2)$, then Stone's formula (See e.g. \cite[Chapter XIV]{H83}.)
\begin{equation}\label{eq:Stone}
dE_\lambda=\frac{R^+_\lambda-R^-_\lambda}{2\pi i}
\end{equation}
immediately implies the same norm bounds from $L^2(\pO)$ to $L^2(\Omega)$. Moreover, the sharp examples of these bounds on spectral measure operators automatically provides the sharpness for single layer potentials.

\begin{proposition}[Boundedness and sharpness of spectral measure operators]\label{prop:dE}\hfill
\begin{enumerate}[(i).]
\item In a general domain $\Omega$,
$$\|dE_\lambda(ud\sigma)\|_{L^2(\Omega)}\le c\lambda^{-\frac34}\|u\|_{L^2(\partial\Omega)},$$
and the norm is sharp if the boundary $\pO$ contains a flat piece.
\item If $\partial\Omega$ is curved, then
$$\|dE_\lambda(ud\sigma)\|_{L^2(\Omega)}\le c\lambda^{-\frac56}\|u\|_{L^2(\partial\Omega)}.$$
and the norms are sharp if $\Omega$ is an annulus.
\end{enumerate}
\end{proposition}

Notice that the kernel of $dE_\lambda$ 
$$\tilde K_\lambda=\frac{K^+_\lambda-K^-_\lambda}{2\pi i}$$
satisfies $(-\Delta-\lambda^2)\tilde K_\lambda=0$.  To estimate the norm of 
$$dE_\lambda(\cdot\,d\sigma):L^2(\pO)\to L^2(\Omega),$$
we consider the adjoint operator $(dE_\lambda)^\star$ with the same norm. In particular, the estimates of $dE_\lambda$ in Proposition \ref{prop:dE} is equivalent to
\begin{equation}\|(dE_\lambda)^\star(u)\|_{L^2(\pO)}\le
\begin{cases}
c\lambda^{-\frac34}\|u\|_{L^2(\Omega)}&\quad\text{in general domains};\\
c\lambda^{-\frac56}\|u\|_{L^2(\Omega)}&\quad\text{if $\pO$ is curved}.
\end{cases}\label{dEstarest}\end{equation}
We provide a proof of \eqref{dEstarest} here. It also motivates the strategy of the main proof of Theorem \ref{thm:SLP}. Write $v=(dE_\lambda)^\star(u)$, then $v$ is an eigenfunction in $\R^n$. In particular, let $\Omega_1$ be a compact set such that $\Omega\Subset\Omega_1\Subset\R^n$. Then $v$ is an eigenfunction in $\Omega_1$, and $\pO$ can be regarded as an interior hypersurface in $\Omega_1$. It is a classical result in scattering theory that $(dE_\lambda)^\star:L^2(\Omega)\to L^2(\Omega_1)$ is bounded. In fact, from the semiclassical Fourier integral operator theory, (See e.g. \cite{RT}.) we have
\begin{equation}\label{eq:FIO}
\|v\|_{L^2(\Omega_1)}\le c\lambda^{-1}\|u\|_{L^2(\Omega)}.
\end{equation}
Since $\pO\subset\Omega_1$, we can use the interior hypersurface restriction estimate for $v$ from \cite{BGT, Hu}:
$$\|v\|_{L^2(\pO)}\le c\lambda^\frac14\|v\|_{L^2(\Omega)}\leq{}c\lambda^{-\frac34}\|u\|_{L^{2}(\Omega)}$$
on general $\pO$, and
$$\|v\|_{L^2(\pO)}\le c\lambda^\frac16\|v\|_{L^2(\Omega)}\leq{}c\lambda^{-\frac56}\|u\|_{L^{2}(\Omega)}$$
if $\pO$ is curved. 

\subsection*{Connection with interior hypersurface estimates of quasimodes and strategy of the proofs}
Consider the adjoint operator
$$(S^+_\lambda)^\star:L^2(\Omega)\to L^2(\pO).$$
Unlike $(dE_\lambda)^{\star}(u)$, $(S^{+}_{\lambda})^{\star}(u)$ is not an eigenfunction on the whole space $\R^n$. The failure of $(S^{+}_{\lambda})^{\star}(u)$ to be an eigenfunction arises from a singularity in its kernel at the diagonal.    Our strategy is therefore to divide the kernel into near-diagonal and off-diagonal parts. The near-diagonal part admits better bound than required, and the off-diagonal part while not an exact eigenfunction is a good approximate eigenfunction (or quasimode) which can be treated within the semiclassical framework.

To use the the semiclassical framework we set $h=\lambda^{-1}$, then for
 $$p(x,hD)=(-h^{2}\Delta-1)$$
 Laplacian eigenfunctions satisfy $p(x,hD)u=0$.  Given any function $v$ we may measure the quasimode error 
$$E[v]=(-h^{2}\Delta-1)v$$
See \cite[Section 7.4.1]{Z} for more details on quasimodes. 
Restrictions of  quasimodes to hypersurfaces are studied in  \cite{T} and  \cite{HTacy}.    Therefore we are able to reduce the problem of operator norm estimates to that of estimating the quasimode error of the off-diagonal contribution.  

A similar strategy applies to double layer potential $(D^+_\lambda)^\star$, as the off-diagonal part on the boundary resembles the normal derivative of an $O_{L^2}(h)$ quasimode, and therefore we can use the result on Neumann data restriction estimates in \cite{T14}. 

\subsection*{Semiclassical interpretation}
Here we provide a semiclassical description of the problem and our approach. While this description is not strictly necessary to prove Theorems \ref{thm:SLP} and \ref{thm:DLP} it allows us to develop a useful heuristic that gives us insight into the role of geometry in these estimates.  In the semiclassical setting (with $h=\lambda^{-1}$) Laplacian eigenfunctions are solutions to
$$p(x,hD)u=(-h^{2}\Delta-1)u=0,$$
where $p(x,hD)$ has symbol $p(x,\xi)=|\xi|^{2}-1$. Similarly \eqref{eq:fund} can be written as
$$p(x,hD)K^{+}_{h^{-1}}(x)=h^{2}\delta(x).$$
H\"ormander's theory on propagation of singularities asserts that
\begin{equation}\label{eq:WF}
WF_h(K)\setminus WF_h(g)\subset\{(x,\xi)\in T^*\R^n:p(x,\xi)=0\}=\{(x,\xi)\in T^*\R:|\xi|=1\}=S^*\R^n,
\end{equation}
which means that $WF_h(K)\setminus WF_h(g)$ is in the bicharacteristic variety of $p$, and furthermore it is invariant under the Hamiltonian flow $\Phi^t$ of $p$. Here, $WF_h$ is the semiclassical wavefront set, $T^*\R^n$ and $S^*\R^n$ are the cotangent and cosphere bundles of $\R^n$, and $\Phi^t$ of $p$ is the geodesic flow in $\R^n$. See \cite{H85, Z} for a complete discussion of the theory.

This framework provides an heuristic to understand the improvement in Theorem \ref{thm:SLP} for curved domains. 

The dominating singularities of $S^{+}_\lambda$ propagate through the bicharacteristic flowout, i.e., the geodesic flow. Since $f$ is supported on $\pO$, the dominating singularities propagate along the lines tangent to the boundary. Therefore it is natural to expect worse estimates in the flat case where tangent lines coincide with the boundary.

Next we consider the semiclassical description of the outgoing and incoming resolvents $R_\lambda^\pm$ and the spectral measure $dE_\lambda$. From the intersecting Lagrangian distribution theory introduced in \cite{MU} (for the semiclassical version see \cite[Appendix A]{Galk}), $R^+_\lambda$ is an intersecting Lagrangian distribution associated to two Lagrangian submanifolds \cite[Theorem 3]{Galk}:
\begin{itemize}
\item The conormal bundle to the diagonal,
$$L_1=\{(x,\xi,y,\eta)\in T^*\R^n\times T^*\R^n:x=y,\xi=\eta\},$$ 
which is the lift of $WF_h(g)$ in \eqref{eq:WF} from in $T^*\R^n$ to in $T^*\R^n\times T^*\R^n$;
\item The bicharacteristic flowout $\Phi^t$ in the positive direction from the intersection of $L_1$ and the bicharacteristic variety $S^*\R^n\times S^*\R^n$,
$$L_2=\{(x,\xi,y,\eta)\in T^*\R^n\times T^*\R^n:\xi=\eta, |\eta|=1,x=y+t\eta,t\ge0\}.$$ 
\end{itemize} 

$R^-_\lambda$ is the same except it would be the bicharacteristic flowout in the negative direction. When we subtract them in \eqref{eq:Stone}, the diagonal part cancels and $dE_\lambda$ is associated to the flowout $L_0$ in both directions, from the intersection of $L_1$ and the characteristic variety:
$$L_0=\{(x,\xi,y,\eta)\in T^*\R^n\times T^*\R^n:\xi=\eta,|\eta|=1,x=y+t\eta,t\in\R\}.$$  

The above characterization can also been seen in the proof of Theorem \ref{thm:SLP} in Section 2, as we cut the kernel into near-diagonal and off-diagonal parts, which correspond to the two Lagrangian submanifolds.

\subsection*{Organisation of the paper}
In Section \ref{sec:SLP}, we prove Theorem \ref{thm:SLP} and in Section \ref{sec:DLP} we prove Theorem \ref{thm:DLP}. In Section \ref{sharpness}, we show that all of these estimates are essentially sharp, and then give some further remarks concerning the relation between these bounds and the convexity of the domain. In the appendix, we prove the mapping norms of semiclassical layer operators and show that the estimates are nearly sharp.

Throughout this paper, $A\lesssim B$ ($A\gtrsim B$) means $A\le cB$ ($A\ge cB$) for some constant $c$ depending only on the domain, in particular, independent of $\lambda$; $A\approx B$ means $A\lesssim B$ and $B\lesssim A$; the constants $c$ and $C$ may vary from line to line.

\subsection*{Acknowledgements}
The problems considered in this paper were brought to us by Andrew Hassell, we  thank him for the helpful discussion throughout the preparation. We also thank Alex Barnett for informing us the related work of Chandler-Wilde et al. J.G. would like to Maciej Zworski for valuable guidance, Hart Smith for discussion of sharpness of single layer operator estimates, and Andrew Hassell for discussion of sharpness of the double layer operator estimates. J.G. is grateful to the National Science Foundation for support under the National Science Foundation Graduate Research Fellowship Grant No. DGE 1106400 and grant DMS-1201417. X. H. acknowledges the support of the Australian Research Council through Discovery Project DP120102019.

\section{Boundedness of semiclassical single layer potentials}\label{sec:SLP}

We aim to use previously know bounds for restriction of quasimode to hypersurfaces, both in the general and curved cases. We want mapping norm bounds for $S^{+}_{\lambda}:L^{2}(\partial\Omega)\to{}L^{2}(\Omega)$. However to use prior results on restriction of quasimodes we actually study the adjoint operator $(S^{+}_{\lambda})^{\star}:L^{2}(\Omega)\to{}L^{2}(\partial\Omega)$. Now
\begin{equation}(S^{+}_{\lambda})^{\star}u(x)=\int_{\Omega}K^{\star}_{\lambda}(x-y)u(y)dy,\label{Sstar}\end{equation}
where
$$(-\Delta-\lambda^{2})K^{\star}_{\lambda}=\delta(x).$$
Therefore if $R_{\partial\Omega}$ is the restriction operator to the boundary of $\Omega$ we must  prove $L^{2}(\Omega)\to{}L^{2}(\partial\Omega)$ estimates for $R_{\partial\Omega}S^{+}_{\lambda}$. We will do this by constructing an auxiliary quasimode $v$ defined on $\R^{n}$, for which we know the restriction bounds, then the problem reduces to finding the $L^{2}$ norm and the quasimode error of $v$. 

To begin we excise the diagonal of $(S^{+}_{\lambda})^{\star}$ that is let $\zeta:\R^{n}\to{}\R^{+}$ be a smooth cut off function equal to one in $|x-y|\leq{}1$ and supported in $|x-y|\leq{}2$. Then we decompose $(S^{+}_{\lambda})^{\star}$ as
\begin{equation}
(S^{+}_{\lambda})^{\star}=S^{0}+\tilde{S},\label{decomposition}
\end{equation}
where
\begin{equation}S^{0}u(x)=\int_{\Omega}K^{\star}_{\lambda}(x-y)\zeta\left(M^{-1}\lambda(x-y)\right)u(y)dy,\label{S0def}\end{equation}
\begin{equation}
\tilde{S}u(x)=\int_{\Omega}K^{\star}_{\lambda}(x-y)\left(1-\zeta\left(M^{-1}\lambda(x-y)\right)\right)u(y)dy.\label{Stildedef}\end{equation}
We first show that $S^{0}$ has a better $L^{2}(\Omega)\to{}L^{2}(\partial\Omega)$ mapping norm than predicted by Theorem \ref{thm:SLP} and therefore we may focus on  the mapping norm of $\tilde{S}$.

\begin{proposition}\label{S0prop}
Let $S^{0}$ be as defined in \eqref{S0def}, then
\begin{equation}\|S^{0}u\|_{L^{2}(\partial\Omega)}\lesssim{}\lambda^{-\frac32}\|u\|_{L^{2}(\Omega)}.\label{S0bound}\end{equation}
\end{proposition}

\begin{proof}
We use the explicit respresntation of $K_{\lambda}^{\star}$ as a Hankel function. If $n\ge3$, we have that the kernel of $S^{0}$, $K^{0}(x,y)$ has the bounds
$$|K^{0}(x,y)|\leq{}|x-y|^{-(n-2)},$$
and is supported in $|x-y|\leq{}M\lambda^{-1}$. Fixing $x$ we have
$$\|K^{0}(x,\cdot)\|_{L^{1}}\lesssim\int_{0}^{M\lambda^{-1}}r^{-(n-2)}r^{n-1}dr$$
$$\lesssim{}C_{M}\lambda^{-2}.$$
Conversely fixing $y$ we have
$$\|K^0(\cdot,y)\|_{L^{1}}\lesssim\int_{0}^{M\lambda^{-1}}r^{-(n-2)}r^{n-2}dr$$
$$\lesssim{}C_{M}\lambda^{-1}.$$
Therefore by Young's inequality
$$\|S^{0}u\|_{L^{2}(\partial\Omega)}\lesssim{}C_{M}\lambda^{-3/2}\|u\|_{L^{2}(\Omega)}$$
which is better than \eqref{S0bound}.

In $\R^2$, if $|x-y|\leq{}M\lambda^{-1}$
$$|K^{0}(x,y)|\leq{}\log\left(\lambda|x-y|\right)\leq{}C_{\varepsilon}\left(\lambda|x-y|\right)^{-\varepsilon}$$
for any $\varepsilon>0$. The same application of Young's inequality implies
$$\|S^{0}u\|_{L^{2}(\partial\Omega)}\lesssim{}C_{M,\varepsilon}\lambda^{-3/2}\|u\|_{L^{2}(\Omega)}.$$
\end{proof}

We now focus on the operator $\tilde{S}$. Let
$$v=\tilde{S}u.$$
We will treat $v$ as a quasimode of $(-h^{2}\Delta-1)$. Accordingly let
$$E[v]=(-h^{2}\Delta-1)v.$$
From \cite{T} and \cite{HTacy}, we know that
\begin{equation}\norm{v}_{L^{2}(\partial\Omega)}\lesssim{}h^{-\frac{1}{4}}\left[\norm{v}_{L^{2}(\R^{n})}+h^{-1}\norm{E[v]}_{L^{2}(\R^{n})}\right]\label{generalrestric}\end{equation}
in the general case and
\begin{equation}\norm{v}_{L^{2}(\partial\Omega)}\lesssim{}h^{-\frac{1}{6}}\left[\norm{v}_{L^{2}(\R^{n})}+h^{-1}\norm{E[v]}_{L^{2}(\R^{n})}\right]\label{curverestric}\end{equation}
in the case where $\partial\Omega$ is curved. We can also obtain these bounds from \cite{Ta} by considering the function $e^{\frac{i}{h}t}v$ which is an approximate solution to the wave equation. It is known that $(S^{+}_{\lambda})^{\star}$ has mapping norm $\lambda^{-1}$ from $L^{2}(\Omega)\to{}L^{2}(\R^{n})$, see for example \cite{VZ}. By the arguments of Proposition \ref{S0prop} $S^{0}$ has mapping norm $\lambda^{-2}$ from $L^{2}(\Omega)\to{}L^{2}(\R^{n})$ therefore as 
$$\tilde{S}=(S^{+}_{\lambda})^{\star}-S^{0},$$
$$\norm{v}_{L^{2}(\R^{n})}=\norm{\tilde{S}u}_{L^{2}(\R^{n})}\lesssim\lambda^{-1}\norm{u}_{L^{2}(\Omega)}=h\norm{u}_{L^{2}(\Omega)}.$$
So to obtain Theorem \ref{thm:SLP} it is enough to show that
$$\norm{E[v]}_{L^{2}(\R^{n})}\lesssim{}h^{2}\norm{u}_{L^{2}(\Omega)}.$$
Rescaling this to work in terms of $\lambda$ we require that
$$\norm{(-\Delta-\lambda^{2})v}_{L^{2}(\R^{n})}\lesssim\norm{u}_{L^{2}(\Omega)}.$$
Now
$$v(x)=\int_{\Omega}K^{\star}_{\lambda}(x-y)\left(1-\zeta(M^{-1}\lambda(x-y))\right)u(y)dy,$$
where
$$(-\Delta-\lambda^{2})K^{\star}_{\lambda}=\delta.$$
So applying the operator $(-\Delta-\lambda^{2})$ we have
\begin{equation}
(-\Delta-\lambda^{2})v=\int_{\Omega}\left(1-\zeta(M^{-1}\lambda(x-y))\right)[(-\Delta_{x}-\lambda^{2})K^{\star}_{\lambda}(x-y)]u(y)dy+\tilde{E}u\label{quasimode},
\end{equation}
\begin{multline}
\tilde{E}u=\int_{\Omega}[-K^{\star}_{\lambda}(x-y)\Delta_{x}\left(1-\zeta(M^{-1}\lambda(x-y))\right)\\
-2\nabla_{x}\left(1-\zeta(M^{-1}\lambda(x-y))\right)\cdot\nabla_{x}K^{\star}_{\lambda}(x-y)]u(y)dy.\label{Edef}
\end{multline}

The first term in \eqref{quasimode} is zero as the support of $\left(1-\zeta(M^{-1}\lambda(x-y))\right)$ is bounded away from the diagonal $x=y$. The second term is the error term and has kernel supported in $M\lambda^{-1}\leq|x-y|\leq{}2M\lambda^{-1}$. It therefore suffices to show that
$$\norm{\tilde{E}u}_{L^{2}(\R^{n})}\lesssim{}\norm{u}_{L^{2}(\Omega)}.$$

\begin{proposition}
If $\tilde{E}$ is given by \eqref{Edef}, then
$$\norm{\tilde{E}u}_{L^{2}(\R^n)}\lesssim{}\norm{u}_{L^{2}(\Omega)}.$$
\end{proposition}

\begin{proof}
This is similar to the proof of Proposition \ref{S0prop}. By choosing $M$ large enough we may assume that the argument of the Hankel function $\lambda|x-y|$ is large and therefore 
$$|H^{(2)}_{\beta}(\lambda|x-y|)|\leq{}\lambda^{-\frac{1}{2}}|x-y|^{-\frac{1}{2}}$$
for any $\beta$. Therefore on the support of the kernel of $E$ we have
$$|K_{\lambda}(x-y)|\lesssim{}\lambda^{n-2},$$
$$|\nabla_{x}K_{\lambda}(x-y)|\lesssim\lambda^{n-1},$$
$$|\nabla_{x}\left(1-\zeta(M^{-1}\lambda(x-y))\right)|\lesssim{}\lambda,$$
$$|\Delta_{x}\left(1-\zeta(M^{-1}\lambda(x-y))\right)|\lesssim{}\lambda^2.$$
Therefore we may write
$$\tilde{E}u=\int\widetilde{K}(x-y)u(y)dy,$$
where $|\widetilde{K}(x-y)|\lesssim\lambda^{n}$ and is supported on $M\lambda^{-1}\leq{}|x-y|\leq{}2M\lambda^{-1}$. Now
$$\norm{\widetilde{K}(\cdot)}_{L^{1}}\lesssim\lambda^{n}\int_{0}^{2M\lambda^{-1}}r^{n-1}dr\lesssim{}1.$$
Therefore by Young's inequality
$$\norm{\tilde{E}u}_{L^{2}(\R^{n})}\lesssim{}\norm{u}_{L^{2}(\Omega)}$$
as required.
\end{proof}

\section{Boundedness of semiclassical double layer potentials}\label{sec:DLP}

We now address the mapping norms of the double layer potential
\begin{equation}D^{+}_{\lambda}u=\int_{\partial\Omega}\partial_{\nu_y} K_\lambda(x-y)u(y)d\sigma_y.\label{DLPdef}\end{equation}
We proceed in a similar fashion as the proof for the single layer potential working instead with the adjoint operator $(D^{+}_{\lambda})^{\star}$. Let $\zeta:\R^{n}\to{}\R$ be a smooth cut off function equal to one in $|x|\leq{}1$ and supported in $|x|\leq{}2$. Then we decompose $D_{\lambda}^{\star}$ as
\begin{equation}(D^{+}_{\lambda})^{\star}=D_{0}+\widetilde{D}\label{DLPdecomp}\end{equation}
where
\begin{equation}D_{0}u=\int_{\Omega}\partial_{\nu_x} K^{\star}_\lambda(x-y)\zeta\left(M^{-1}\lambda(x-y)\right)u(y)dy\label{D0def}\end{equation}
and
\begin{equation}
\widetilde{D}u=\int_{\Omega}\partial_{\nu_x} K^{\star}_\lambda(x-y)\left[1-\zeta\left(M^{-1}(x-y)\right)\right]u(y)dy.\label{Dtildedef}\end{equation}
Similar to the single layer potential case we will treat $D_{0}$ by Young's inequality and $\widetilde{D}$ by quasimode methods.

\begin{proposition}\label{D0prop}
Let $D_{0}$ be as defined in \eqref{D0def} then
$$\|D_{0}u\|_{L^{2}(\partial\Omega)}\lesssim{}\lambda^{-1/2}\|u\|_{L^{2}(\Omega)}$$
\end{proposition}

\begin{proof}
We have that
\begin{multline*}
\partial_{\nu_x}K^{\star}_{\lambda}(x-y)=\frac{i}{4}\left(\frac{x-y}{|x-y|}\cdot\nu_{x}\right)\Biggl[-\left(\frac{\lambda}{2\pi|x-y|}\right)^{\frac{n-2}{2}}\lambda{}H^{(2)}_{\frac{n}{2}}(\lambda|x-y|)-\\
\frac{n-2}{2|x-y|}\left(\frac{\lambda}{2\pi|x-y|}\right)^{\frac{n-2}{2}}H^{(2)}_{\frac{n-2}{2}}(\lambda|x-y|)-
\lambda^{-1} \left(\frac{\lambda}{2\pi|x-y|}\right)^{\frac{n}{2}}H^{(2)}_{\frac{n-2}{2}}(\lambda|x-y|)\Biggr]\end{multline*}
Therefore on the support of the kernel of $D_{0}$ we have that
$$|\partial_{\nu_x}K^{\star}_{\lambda}(x-y)|\leq{}|x-y|^{-(n-1)}$$
We cannot directly apply Young's inequality as $\|K(\cdot,y)\|_{L^{1}}$ is not bounded. However if we decompose dyadically we may use Young's inequality on each piece and, since $\|K(x,\cdot)\|_{L^{1}}$ is much better than $O(1)$, recover something summable.
Accordingly we write
$$D_{0}=\sum_{j=0}^{\infty}D_{0}^{j}$$
where
\begin{equation}
D_{0}^{j}u=\int_{\Omega_{j}}\partial_{\nu_y} K^{\star}_\lambda(x-y)\zeta\left(M^{-1}\lambda(x-y)\right)u(y)dy\label{D0idef}
\end{equation}
$$\Omega_{j}=\Omega\cap\{y\mid{}2^{-j}M\lambda^{-1}\leq|x-y|\leq{}2^{-j+1}M\lambda^{-1}\}.$$
Now applying Young's inequality to each $D_{0}^{i}$ we have
$$\|D_{0}^{j}u\|_{L^{2}(\partial\Omega)}\lesssim{}(2^{j}\lambda)^{n-1}\cdot{}(2^{-j}\lambda^{-1})^{\frac{n}{2}}(2^{-j}\lambda^{-1})^{\frac{n-1}{2}}\|u\|_{L^{2}(\Omega)}$$
$$\lesssim{}\lambda^{-1/2}2^{-j/2}\|u\|_{L^{2}(\Omega)}$$
and therefore
$$\|D_{0}u\|_{L^{2}(\partial\Omega)}\lesssim\lambda^{-1/2}\|u\|_{L^{2}(\Omega)}$$
as claimed.

\end{proof}

\begin{proposition}
Let $\widetilde{D}$ be given by \eqref{Dtildedef} then
$$\|\widetilde{D}u\|_{L^{2}(\partial\Omega)}\lesssim{}\|u\|_{L^{2}(\Omega)}$$

\end{proposition}

\begin{proof}
We note that if we define the auxiliary function $w$ by
\begin{equation}
w=\widetilde{D}u
\label{wdef}
\end{equation}
then 
\begin{equation}w=\partial_{\nu}v+\tilde{E}u\label{wnormderiv}\end{equation}
where 
$$\tilde{E}u=\frac{\lambda}{M}\int_{\Omega}K^{\star}_{\lambda}(x-y)\partial_{\nu_y}\zeta\left(M^{-1}\lambda(x-y)\right)u(y)dy$$
and $v$ is the quasimode
$$v=\tilde{S}u$$
introduced in the proof of Theorem \ref{thm:SLP}.

From  \cite{T14} and \cite{Ta} we know that normal derivatives of quasimodes enjoy the hypersurface restriction bound
$$\|\partial_{\nu}v\|_{L^{2}(\partial\Omega)}\lesssim{}\lambda\norm{v}_{L^{2}(\Omega)}$$
therefore by the $L^{2}(\Omega)\to{}L^{2}(\R^{n})$ mapping properties of the single layer potential
\begin{equation}
\|\partial_{\nu}v\|_{L^{2}(\partial\Omega)}\lesssim{}\norm{u}_{L^{2}(\Omega)}.\label{normderivbnd}
\end{equation}
So we can restrict our attention to $\tilde{E}u$. We write
$$\tilde{E}u=\int_{\Omega}\tilde{E}(x,y)u(y)dy$$
and note by Young's inequality
$$\|\tilde{E}u\|_{L^{2}(\partial\Omega)}\lesssim{}\sup_{x,y}\|\tilde{E}(x,\cdot)\|_{L^{1}}^{1/2}\|\tilde{E}(\cdot,y)\|_{L^{1}}^{1/2}\|u\|_{L^{2}(\Omega)}.$$
On the support of $\tilde{E}(x,y)$ we have that
$$|K^{\star}(x-y)|\lesssim{}\lambda^{n-2}\quad{}n\neq{}2$$
$$|K^{\star}(x-y)|\lesssim\log\lambda{}\quad{}n=2$$
Therefore for $n\neq{}2$ we have
$$\sup_{x}\|\tilde{E}(x,\cdot)\|\lesssim{}\lambda^{n-1}\cdot{}\lambda^{-n}\lesssim\lambda^{-1}$$
and
$$\sup_{y}\|\tilde{E}(\cdot,y)\|\lesssim{}\lambda^{n-1}\cdot{}\lambda^{-(n-1)}\lesssim{}1.$$
For $n=2$
$$\sup_{x}\|\tilde{E}(x,\cdot)\|\lesssim{}\lambda\log\lambda\cdot{}\lambda^{-2}\lesssim\lambda^{-1}\log\lambda$$
and
$$\sup_{y}\|\tilde{E}(\cdot,y)\|\lesssim{}\lambda\log\lambda\cdot{}\lambda^{-1}\lesssim\log\lambda$$
so for any $\epsilon>0$
$$\|\tilde{E}u\|_{L^{2}(\partial\Omega)}\lesssim{}\lambda^{\frac{-1+\epsilon}{2}}\lambda^{\frac{\epsilon}{2}}\|u\|_{L^{2}(\Omega)}$$
and therefore setting $\epsilon<1/2$
$$\|\tilde{E}u\|_{L^{2}(\partial\Omega)}\lesssim{}\|u\|_{L^{2}(\Omega)}$$
as required.

\end{proof}

\section{Sharp examples and further remarks}\label{sharpness}

In this section, we construct examples to show that the estimates in Theorems \ref{thm:SLP} and \ref{thm:DLP} are in fact sharp. That is, we prove that lower bounds hold in some domains for some sequences of functions. Furthermore, we make some remarks on single layer potentials and spectral measure operators in strictly convex domains.

\subsection{Sharpness of Theorem \ref{thm:SLP}: Semiclassical single layer potentials in general domains}\label{sec:SLPsharpgeneral}
In the view of the Stone's formula \eqref{eq:Stone} and Proposition \ref{prop:dE}, we only need to prove the sharpness of $dE_\lambda$, and then the sharpness of $S_\lambda$ follows immediately. In fact, we construct functions $\{f_{\lambda}\}$ on the square $[-1,1]^{n-1}$ such that
\begin{equation}\label{square}
\frac{\|dE_\lambda(f_\lambda d\sigma)\|_{L^2([-1,1]^{n-1}\times[0,1])}}{\|f_\lambda\|_{L^2([-1,1]^{n-1})}}\ge c\lambda^{-\frac34}.
\end{equation}

Throughout this subsection, we denote $x=(x',x_n)\in\R^n$ where $x'\in\R^{n-1}$ and $x_n\in\R$. We develop our sharp example through a series of lemmas. First we observe the following fact.

\begin{lemma}\label{constructf}
Write $\Omega'=[-1,1]\times\cdots\times[-1,1]\subset\R^{n-1}$, for $\lambda\ge0$ there exists an $L^2$ normalized function $f_\lambda$ such that 
\begin{enumerate}
\item $supp\,f_\lambda\subset\Omega'$,
\item $\hat f_\lambda\ge0$,
\item $\hat f_\lambda(\xi')\ge c_1$ if $|\xi'-\eta'_\lambda|\le1$ for $\eta'_\lambda=(\lambda,0,...,0)$ and some positive constant $c_1$ depending only on the dimension.
\end{enumerate}
\end{lemma}

\begin{proof}[Proof of Lemma \ref{constructf}]
Fix a Schwartz function $\varphi$ such that $\hat\varphi\ge0$ and $\hat\varphi=1$ in $\Omega'$. Let $f_0=\varphi\chi_{\Omega'}/\|\varphi\chi_{\Omega'}\|_{L^2(\R^{n-1})}$, obviously $supp\,f_0\subset\Omega'$, and we verify that $f_0$ also satisfies (2) and (3) above. Since $\chi_{\Omega'}$ is even, and
$$\hat\chi_{\Omega'}(\xi')=\prod_{i=1}^{n-1}\frac{2\sin(\xi_i)}{\xi_i},$$
we obtain $\hat f_0=c\hat\varphi\ast\hat\chi_{\Omega'}\ge0$. Now we compute
$$\hat f_0(\xi')=c\int_{\R^{n-1}}\hat\varphi(\eta')\hat\chi_{\Omega'}(\xi'-\eta')d\eta'\ge c\int_{|\eta'|\le1}\hat\chi_{\Omega'}(\xi'-\eta')d\eta'\ge c_1,$$
if $|\xi'-\eta'_0|\le1$. Therefore if we set $f_\lambda(x')=e^{ix'\cdot\eta'_\lambda}f_0(x')$, we have constructed of function that satisfies all the required condtions.
\end{proof}

Now denote $\Omega=\Omega'\times[0,1]\subset\R^n$, and let $supp\,f_\lambda\subset\Omega'\times\{x_n=0\}$. We claim that $f_{\lambda}$ is our desired sharp example. That is
$$\|dE_{\lambda}(f_{\lambda}d\sigma)\|_{L^{2}(\Omega)}=\|dE_{\lambda}(fd\sigma)\chi_{\Omega}\|_{L^{2}(\R^{n})}\geq{}c\lambda^{-\frac34}.$$

To facilitate our calculation we will replace $\chi_{\Omega}$ with a function $g$ designed to make calculation on the Fourier transform side easy. We need the following lemma.

\begin{lemma}\label{constructg}
There exists a function $g$ such that 
\begin{enumerate}
\item $supp\,g\subset\Omega$,
\item $0\le g\le c_2$ for some constant $c_2>0$ depending only on the dimension,
\item $\hat g(\xi)\ge c_3$ in $\{|\xi|\le c_4\}$ for some positive constants $c_3$ and $c_4\le\frac12$ depending only on the dimension.
\end{enumerate}
\end{lemma}

\begin{proof}[Proof of Lemma \ref{constructg}]
Fix $z=(0,...,0,\frac12)\in\Omega$, and write $\varphi=\chi_{|x-z|\le\frac14}$. Let $g=\varphi\ast\varphi$, then both (1) and (2) above are satisfied, and
$$\hat g(0)=[\hat\varphi(0)]^2=\left[\int_{\R^n}\varphi(x)dx\right]>0.$$
Thus, $\hat g(\xi)\ge c_3$ in $\{|\xi|\le c_4\}$ because $\hat g$ is continuous.
\end{proof}

In order to evaluate $\|dE_\lambda(f_\lambda d\sigma)\|_{L^2(\Omega)}$, notice that $\widehat{f_\lambda d\sigma}(\xi)=\hat f_\lambda(\xi')$ and
$$[dE_\lambda(f_\lambda d\sigma)]^\wedge(\xi)=\delta(|\xi|^2-\lambda^2)\widehat{f_\lambda d\sigma}(\xi)=\frac{\hat f_\lambda(\xi')d\mu}{2\lambda},$$ 
where $d\mu$ is the surface measure on $\{|\xi|=\lambda\}$. Using the function $g$ constructed in Lemma \ref{constructg}, for any $\xi$ such that 
$$\xi\in\left\{0\le\lambda-|\xi|\le\frac{c_4}{2}\text{ and }|\xi'-\eta'_\lambda|\le\frac12\right\}:=G_\lambda,$$
we have 
$$[dE_\lambda(f_\lambda d\sigma)]^\wedge\ast\hat g(\xi)\ge\frac{2c_3}{\lambda}\int_{\{|\eta|=\lambda\}\cap\{|\eta-\xi|\le c_4\}}\hat f_\lambda(\eta')d\mu\ge\frac{c\cdot c_1\cdot c_3\cdot c^{n-1}_4}{\lambda},$$
in which we use the geometric fact that the area measure $|\{|\eta|=\lambda\}\cap\{|\eta-\xi|\le c_4\}|\sim c^{n-1}_4$ if $\xi$ is a fixed point near the sphere with $0\le\lambda-|\xi|\le\frac{c_4}{2}<\frac14$. Recall that $\eta'_\lambda=(\lambda,0,...,0)$, then the volume measure of $G_\lambda$ has
$$|G_\lambda|=\left|\left\{0\le\lambda-|\xi|\le\frac{c_4}{2}\text{ and }|\xi'-\eta'_\lambda|\le\frac12\right\}\right|\sim\sqrt\lambda.$$
Therefore,
$$\|dE_\lambda(f_\lambda d\sigma)g\|_{L^2(\R^n)}=c\left\|[dE_\lambda(f_\lambda d\sigma)]^\wedge\ast\hat g\right\|_{L^2(\R^n)}\ge c\lambda^{-1}|G_\lambda|^\frac12\ge c\lambda^{-\frac34}.$$
Now, since $g$ is supported in $\Omega$ and bounded from above,
$$\|dE_\lambda(f_\lambda d\sigma)\|_{L^2(\Omega)}=\|dE_\lambda(f_\lambda d\sigma)\chi_\Omega\|_{L^2(\R^n)}\ge c^{-1}_2\|dE_\lambda(f_\lambda d\sigma)g\|_{L^2(\R^n)}\ge c\lambda^{-\frac34},$$
and we have obtained \eqref{square}.

\subsection{Sharpness of Theorem \ref{thm:SLP}: Semiclassical single layer potentials in curved domains}\label{sec:SLPsharpcurved}

As in Section \ref{sec:SLPsharpgeneral}, we only need to prove the sharpness of $dE_\lambda$. We construct functions $\{f_{\lambda}\}$ such that
\begin{equation}\label{annulus}
\frac{\|dE_\lambda(f_\lambda d\sigma)\|_{L^2(\Omega)}}{\|f_\lambda\|_{L^2(\partial\Omega)}}\ge c\lambda^{-\frac56}
\end{equation}
where $\Omega$ is annulus.

Let $B_1=\{x\in\R^2,|x|<1\}$, $B_2=\{x\in\R^2,|x|<2\}$, and $\Omega=\{x\in\R^2,1<|x|<2\}$. We work in polar coordinates $(r,\theta)$ and set
$$f_k(x)=e^{ik\theta}\in L^2(\partial\Omega).$$
Then
$$u(x)=dE_\lambda(f_k d\sigma)(x)=aJ_k(\lambda r)e^{ik\theta},$$
in which $J_k$ is the Bessel function of the first kind and order $k$. We pick $\lambda=j_{k,1}$ as the first positive zero of $J_k$. Then $u$ solves the Dirichlet boundary value problem
$$\begin{cases}
-\Delta u=\lambda^2u & \text{ in }B_1,\\
u=0,\ \partial_ru=e^{ik\theta}& \text{ on }\partial B_1.
\end{cases}$$
We need to show that
$$\|u\|_{L^{2}(\Omega)}\geq{}c\lambda^{-\frac56}.$$
From \cite[Section 10.21.40]{DLMF}, $J'_k(\lambda)=O(k^{-\frac23})$. Thus
\begin{equation}\label{asympa}
a=\frac{1}{\lambda J'_k(\lambda)}=O(k^{-\frac13}).
\end{equation}
We also have
\begin{equation}\label{asymplambda}
\lambda=k+c_5k^\frac13+O(k^{-\frac13}),
\end{equation}
where $c_5=1.86...$ is an independent constant. Since $(\Delta+\lambda^2)u=0$,
$$\left[\partial_r^2+\frac1r\partial_r+\left(\lambda^2-\frac{k^2}{r^2}\right)\right]u=0.$$
Furthermore, $\partial_ru(x)=a\lambda J'_k(\lambda r)e^{ik\theta}$ and $\partial^2_ru(x)=a\lambda^2J''_k(\lambda r)e^{ik\theta}$. To evaluate $\|u\|_{L^2(\Omega)}$, notice that in $\R^n$,
$$\Delta=\partial_r^2+\frac{n-1}{r}\partial_r+\frac{1}{r^2}\Delta_{\mathbb S^{n-1}},$$
in which $\Delta_{\mathbb S^{n-1}}$ is the Laplacian on the sphere $\mathbb S^{n-1}$. Then the commutator
\begin{eqnarray*}
[\Delta,r\partial_r]&=&\left[\partial_r^2+\frac{n-1}{r}\partial_r+\frac{1}{r^2}\Delta_{\mathbb S^{n-1}},r\partial_r\right]\\
&=&\left[\partial_r^2,r\partial_r\right]+\left[\frac{n-1}{r}\partial_r,r\partial_r\right]+\left[\frac{1}{r^2}\Delta_{\mathbb S^{n-1}},r\partial_r\right]\\
&=&2\partial_r^2+\frac{2(n-1)}{r}\partial_r+\frac{2}{r^2}\Delta_{\mathbb S^{n-1}}\\
&=&2\Delta.
\end{eqnarray*}
Using the above facts and Green's formula, we have
\begin{eqnarray*}
-2\lambda^2\int_{|x|<R}|u|^2&=&2\int_{|x|<R}\Delta u\cdot\bar u=\int_{|x|<R}[\Delta,r\partial_r]u\cdot\bar u\\
&=&\int_{|x|<R}[\Delta+\lambda^2,r\partial_r]u\cdot\bar u\\
&=&\int_{|x|<R}(\Delta+\lambda^2)(r\partial_ru)\cdot\bar u-r\partial_ru\cdot(\Delta+\lambda^2)\bar u\\
&=&\int_{|x|=R}\partial_r(r\partial_ru)\cdot\bar u-r\partial_ru\cdot\partial_r\bar u\\
&=&\int_{|x|=R}\partial_ru\cdot\bar u+r\partial_r^2u\cdot\bar u-r|\partial_ru|^2\\
&=&\int_{|x|=R}[-r(\lambda^2-k^2r^{-2})]u\cdot\bar u-r|\partial_ru|^2\\
&=&-2a^2\pi R^2\left[(\lambda^2-k^2R^{-2})(J_k(\lambda R))^2+\lambda^2(J'_k(\lambda R))^2\right],
\end{eqnarray*}
which implies
\begin{equation}\label{L2}
\int_{|x|<R}|u|^2=a^2\pi R^2\left[\left(1-\frac{k^2}{\lambda^2R^2}\right)(J_k(\lambda R))^2+(J'_k(\lambda R))^2\right].
\end{equation}
If $R=1$, then note that $\lambda$ is a zero of $J_k$. As \eqref{asympa} gives $|a|\sim\lambda^{-\frac13}$  and $J'_k(\lambda)=J_k'(j_{k,1})\sim\lambda^{-\frac23}$ in \cite[Section 10.21.40]{DLMF},
\begin{equation}\label{L2B1}
\|u\|_{L^2(B_1)}=\left(\int_{|x|<1}|u|^2\right)^\frac12=|a|\sqrt\pi|J'_k(\lambda)|=c\lambda^{-1}.
\end{equation}
If $R=2$, then
\begin{equation}\label{L2B2}
\|u\|_{L^2(B_2)}=2|a|\sqrt\pi\left[\left(1-\frac{k^2}{4\lambda^2}\right)(J_k(2\lambda ))^2+(J'_k(2\lambda))^2\right]^\frac12\ge c\lambda^{-\frac56}.
\end{equation}
Here, we use the asymptotic expansions of Bessel functions for large orders in \cite{DLMF}
$$J_k(k\sec\beta)\sim\left(\frac{1}{k\tan\beta}\right)^\frac12\cos\left(k\tan\beta-k\beta-\frac14\pi\right),$$
and
$$J'_k(k\sec\beta)\sim\left(\frac{\sin(2\beta)}{k}\right)^\frac12\sin\left(k\tan\beta-k\beta-\frac14\pi\right),$$
in which $\sec\beta=2\lambda/k\to2$ in the view of \eqref{asymplambda}, and thus $\beta\sim\frac\pi3$. Therefore,
$$\|u\|_{L^2(\Omega)}=\|u\|_{L^2(B_2)}-\|u\|_{L^2(B_1)}\ge c\lambda^{-\frac56},$$
as required.

\subsection{Further remark: Semiclassical single layer potentials in strictly convex domains}

If we choose $R=1+\varepsilon$ in \eqref{L2} for any fixed $\varepsilon>0$, then
\begin{equation}\tag{\ref{L2B2}'}
\|u\|_{L^2(B_{1+\varepsilon})}\ge c\lambda^{-\frac56}
\end{equation}
is valid when $k$, and therefore $\lambda$, is large. (We can argue similarly by setting $\beta$ asymptotically fixed depending only on $\varepsilon$.) Comparing \eqref{L2B1} and \eqref{L2B2}', we see that the $L^2$ norm of $u$ is essentially concentrated outside the disc. As discussed in the Introduction, this is because the estimates are dominated by the semiclassical singularities of $u$, which propagate along the tangent lines of the circle.  All such lines lie outside of $B_1$. Therefore, $\|u\|_{L^{2}(B_{1})}\sim\lambda^{-1}$ is smaller than $\|u\|_{L^{2}(B_{1+\epsilon}\setminus{}B_{1})}\sim\lambda^{-\frac56}$.

However, in the case when the domain is flat as in Section \ref{sec:SLPsharpgeneral}, the tangent lines coincide with the boundary, one may get worse $L^2$ bound for $u$ ($\sim\lambda^{-\frac34}$).

The above observation motivates us to consider the problem in strictly convex domains, in which case all the tangent lines lie outside the domain. Analogously to the unit disc, we make the following conjecture.

\begin{conjecture}\label{conj:SLPconvex}
If $\Omega$ is strictly convex, then
\begin{equation}\label{eq:SLPconvex}
\|S_\lambda(f)\|_{L^2(\Omega)}\le c\lambda^{-1}\|f\|_{L^2(\partial\Omega)},
\end{equation}
and
\begin{equation}\label{eq:dEconvex}
\|dE_\lambda(fd\sigma)\|_{L^2(\Omega)}\le c\lambda^{-1}\|f\|_{L^2(\partial\Omega)}.
\end{equation}
\end{conjecture}

The computation in Section \ref{sec:SLPsharpcurved} already gave the sharp example in this case once one observes \eqref{L2B1}, which says 
$$\frac{\|dE_\lambda(f_\lambda d\sigma)\|_{L^2(\Omega)}}{\|f_\lambda\|_{L^2(\partial\Omega)}}=c\lambda^{-1}$$
is valid in the unit ball for some constant $c$ and functions $\{f_\lambda\}$.

In fact, the estimates in Conjecture \ref{conj:SLPconvex} are sharp in any strictly convex domain: From \cite{BLR, HTao02, HTao10},
$$\|u\|_{L^2(\Omega)}\approx\lambda^{-1}\|\partial_\nu u\|_{L^2(\partial\Omega)},$$ 
where $u$ is a Dirichlet eigenfunction in $\Omega$, therefore $u=S_\lambda(\partial_\nu u)$, and \eqref{eq:SLPconvex} and \eqref{eq:dEconvex} are sharp in all strictly convex domains.

\subsection{Sharpness of Theorem \ref{thm:DLP}: Semiclassical double layer potentials}

We show the sharpness in the unit ball:
\begin{equation}\label{ball}
\frac{\|D_\lambda(f_\lambda d\sigma)\|_{L^2(\Omega)}}{\|f_\lambda\|_{L^2(\partial\Omega)}}\ge c
\end{equation}
for some constant $c$ and functions $\{f_\lambda\}$.

Consider the Neumann eigenfunctions:
$$\begin{cases}
-\Delta u=\lambda^2u & \text{ in }B_1,\\
\partial_\nu u=0,\quad u=e^{ik\theta}& \text{ on }\partial B_1.
\end{cases}$$
Then, adopting the same notations as in Section \ref{sec:SLPsharpcurved}, we have
$$u(x)=D_\lambda(f_k d\sigma)(x)=aJ_k(\lambda r)e^{ik\theta},$$
in which $\lambda=j'_{k,l}$ is the $l$-th zero of $J'_k$, and $a=1/J_k(\lambda)$.

We use identity \eqref{L2} with $R=1$ and $J'_k(\lambda)=0$ to obtain
$$\|u\|_{L^2(B_1)}=\left(\int_{|x|<1}|u|^2\right)^\frac12=\sqrt{\pi\left(1-\frac{k^2}{\lambda^2}\right)}\ge c,$$
by picking $\lambda\sim2k$. In fact, $\lambda=j'_{k,l}\to\infty$ as $l\gg k\to\infty$.

On the other hand, to saturate the inequality \eqref{eq:Tataru}:
$$\|u\|_{L^2(\partial B_1)}\lesssim\lambda^\frac13\|u\|_{L^2(B_1)},$$
we pick $\lambda=j'_{k,1}$ as the first zero of $J'_k$. From \cite[Section 10.21.40]{DLMF}, we have
$$\lambda=j'_{k,1}=k+O(k^\frac13),$$
then
$$\|u\|_{L^2(B_1)}=\sqrt{\pi\left(1-\frac{k^2}{(j'_{k,1})^2}\right)}=O(k^{-\frac13})=O(\lambda^{-\frac13}).$$
As $\|u\|_{L^{2}(\partial{}B_{1})}\sim{}1$ this implies that
$$\|u\|_{L^{2}(\partial{}B_{1}}\geq{}c\lambda^{\frac13}\|u\|_{L^{2}(B_{1})}.$$
See also \cite[Example 7]{HTao02} on the boundary estimates of Neumann eigenfunctions.

\appendix

\newcommand{\la}{\langle}
\renewcommand{\ra}{\rangle}

\section{The Single and Double Layer Operators}

\begin{center}
{\sc By Jeffrey Galkowski}\footnote{J.G.'s address is
Department of Mathematics, University of California, Berkeley, CA 94720-3840, USA,
and e-mail address is \texttt{jeffrey.galkowski@math.berkeley.edu}.}

\end{center}
\vspace*{3mm}

In this appendix, we give high frequency estimate for the double and single layer operators. We do this by adapting the methods of \cite{GS} to the double layer operator. The estimates on single layer operators appear in \cite[Theorem 1.2]{GS}, but we repeat them below for the convenience of the reader. 

We use the same notation as in the prequel. In addition, let $\gamma:H^s(\re^d)\to H^{s-1/2}(\partial\Omega)$, $s>1/2$ denote restriction to $\partial\Omega$. Then, define the \emph{single layer operator},
$$\mathcal{S}_{\lambda}^+:=\gamma S^+_\lambda :L^2(\partial\Omega)\to L^2(\partial\Omega)$$
and the \emph{double layer operator}
$$\Dl:L^2(\partial\Omega)\to L^2(\partial\Omega)$$
where 
$$\Dl f(x)=\int _{\partial\Omega}\partial_{\nu_y}K_{\lambda}(x-y)f(y)dy.$$

\begin{theorem}
\label{thm:main}
Let $\partial\Omega\subset\re^d$ be a finite union of compact embedded $C^\infty$ hypersurfaces. Then there exists $\lambda_0$ such that for $\lambda>\lambda_0$,
\begin{equation}
\label{eqn:optimalFlat}
\|\Sl\|_{L^2(\partial\Omega)\to L^2(\partial\Omega)}\leq C\,\lambda^{-\frac 12}\,\log \lambda\,,\quad \quad
\|\Dl\|_{L^2(\partial\Omega)\to L^2(\partial\Omega)}\leq C\,\lambda^{\frac 14}\,\log\lambda\,.
\end{equation}
Moreover, if $\partial\Omega$ is a finite union of compact subsets of curved $C^\infty$ hypersurfaces, then 
\begin{equation}
\label{eqn:optimalConvex}
\|\Sl\|_{L^2(\partial\Omega)\to L^2(\partial\Omega)}\leq C\,\lambda^{-\frac 23}\,\log\lambda\,,\quad\quad
\|\Dl\|_{L^2(\partial\Omega)\to L^2(\partial\Omega)}\leq C\,\lambda^{\frac 16}\,\log\lambda\,.
\end{equation}
Moreover, modulo the factor $\log \lambda$, all of the above estimates are sharp.
\end{theorem}

Such mapping bounds of layer operators in lower dimension cases ($d=2,3$) have been studied by Chandler-Wilde, Graham, Langdon, and Lindner \cite{C-WGLL}. In particular, they showed the upper bounds $\|\Sl\|_{L^2(\partial\Omega)\to L^2(\partial\Omega)}\le c\lambda^{(d-3)/2}$ and $\|\Dl\|_{L^2(\partial\Omega)\to L^2(\partial\Omega)}\le c\lambda^{(d-1)/2}$. In $2$-$\dim$, they also proved that $\|\Sl\|_{L^2(\partial\Omega)\to L^2(\partial\Omega)}\ge c\lambda^{-1/2}$ if $\pO$ contains a flat piece and $\ge c\lambda^{-2/3}$ if $\pO$ is curved. In $2$-$\dim$ they also show the existence of $\pO$ with $\|\Dl\|_{L^2(\partial\Omega)\to L^2(\partial\Omega)}\ge C\lambda^{1/4}$ and curved $\pO$ with $\|\Dl\|_{L^2(\pO)\to L^2(\pO)}\geq C\lambda^{1/8}$. The special cases of the circle and sphere are studied previously by various authors. We refer to \cite{C-WGLL} for more background in this area. This appendix improves these estimates by giving ``nearly'' sharp bounds in all dimensions. We also point out that the approaches to prove estimates for semiclassical layer potentials and operators are completely different.

 In section \ref{sec:mainProof} we prove the upper bounds in Theorem \ref{thm:main}. In sections \ref{sec:sharpSLO} and \ref{sec:sharpDLO} we show that the exponents on $ \lambda$ in Theorem \ref{thm:main} are sharp. 

\noindent {\bf Remark:} If we impose the condition that $\Omega$ is convex with piecewise smooth, $C^{1,1}$ boundary, then we expect that $\Dl$ is uniformly bounded in $\lambda$ but we do not consider that here.

\subsection{Proof of Theorem \ref{thm:main}}
\label{sec:mainProof}

To analyze the single and double layer operators, we rewrite them in terms of the outgoing free resolvent. In particular, we have that 
$$\Sl=\gamma R_\lambda^+ \gamma^*$$
where $R_\lambda^+$ is the outgoing free resolvent. 

To understand $\Dl$ in relation to the free resolvent, we let $L$ be a vector field with $L|_{\partial\Omega}=\partial_\nu$. Then, for $x_0\in \Omega $, $x\in \partial\Omega$, we have
$$\Dl f(x)=\lim_{x_0\to x}\int R_\lambda^+(x_0-y)(L^*(fd\sigma))(y)-\recip{2}f(x)=\lim_{x_0\to x}R_\lambda^+L^*(fd\sigma)(x_0)-\recip{2}f(x).$$
where $L^*=-L-\text{div}L$ and $R_0$ is the outgoing free resolvent (see for example \cite[Section 5]{SMT}, \cite[Section 7.11]{Taylor}).
Moreover,
$$\Dl:L^2(\partial\Omega)\to L^2(\partial\Omega).$$

Let $\la \cdot \ra :=2+|\cdot|$. 
\begin{lemma}
\label{lem:Q}
Suppose that for $\Gamma\Subset \re^d$ any compact embedded $C^\infty$ hypersurface, and some $\alpha\,,\,\beta>0$, 
\begin{align}
\label{eqn:fourierRestrictionEstimate1}
\int |\widehat{L^*f \delta_\Gamma}|^2(\xi)\delta(|\xi|-r)d\xi&\leq C_{\Gamma}\la r\ra^{2\alpha}\|f\|^2_{L^2(\Gamma)},\\
\label{eqn:fourierRestrictionEstimate2}
\int |\widehat{f \delta_\Gamma}|^2(\xi)\delta(|\xi|-r)d\xi&\leq C_{\Gamma}\la r\ra^{2\beta}\|f\|^2_{L^2(\Gamma)}.
\end{align}
Let $\Gamma_1,\,\Gamma_2\Subset \re^d$ be compact embedded $C^\infty$ hypersurfaces.
Let $L$ be a vector field with $L=\partial_{\nu}$ on $\Gamma_1$ for some choice of normal $\nu$ on $\Gamma_1$ and $\psi\in C_0^\infty(\re)$ with $\psi\equiv 1$ in neighborhood of $0$. Then define for $f\in L^2(\Gamma_1)$, $g\in L^2(\Gamma_2)$
$$
\Qs(f,g):=\int R_\lambda^+(\psi(\lambda^{-1}D)f\delta_{\Gamma_1})\bar{g}\delta_{\Gamma_2}\,,\quad\quad
\Qa(f,g):=\int R_\lambda^+(\psi(\lambda^{-1}D)L^*(f\delta_{\Gamma_1}))\bar{g}\delta_{\Gamma_2}.$$
Then for $\Im\lambda>0$, $|\lambda|\geq  c>0$,
\begin{align}
\label{eqn:lowFreqSingle}
|\Qs(f,g)|&\leq C_{\Gamma_1,\Gamma_2}\la \lambda \ra ^{2\alpha-1}\log \la \lambda\ra\|f\|_{L^2(\Gamma_1)}\|g\|_{L^2(\Gamma_2)}\\
\label{eqn:lowFreqDouble}
|\Qa(f,g)|&\leq C_{\Gamma_1,\Gamma_2,\psi}\la \lambda \ra ^{\alpha +\beta-1}\log \la \lambda\ra\|f\|_{L^2(\Gamma_1)}\|g\|_{L^2(\Gamma_2)}
\end{align}
\end{lemma}
\begin{proof}
We follow \cite{GS} to prove the lemma. First, observe that due to the compact support of $f\delta_{\Gamma}$, \eqref{eqn:fourierRestrictionEstimate1} and \eqref{eqn:fourierRestrictionEstimate2}  imply that for $\Gamma\Subset \re^d$,
\begin{align}
\label{eqn:fourierRestrictionEstimate3}
\int\left|\nabla_{\xi}\,\widehat{L^* f\delta_\Gamma}(\xi)\right|^2\delta(|\xi|-r)&\leq C\,\la r\ra^{2\alpha}\|f\|_{L^2(\Gamma)}^2\,,\\
\label{eqn:fourierRestrictionEstimate4}
\int\left|\nabla_{\xi}\,\widehat{f\delta_\Gamma}(\xi)\right|^2\delta(|\xi|-r)&\leq C\,\la r\ra^{2\beta}\|f\|_{L^2(\Gamma)}^2\,.
\end{align}

Now, $g\delta_{\Gamma_2}\in H^{-\frac 12-\e}(\re^d)$, and 
\begin{equation}
\label{eqn:smoothLeft}
R_\lambda^+(\psi(\lambda^{-1}|D|)L^*(f\delta_{\Gamma_1}))\in C^\infty(\re^d),\quad R_\lambda^+(\psi(\lambda^{-1}|D|))f\delta_{\Gamma_1})\in C^\infty(\re^d).
\end{equation}

We only consider $|\lambda|\geq c>0$ to avoid considering the low frequency divergence in low dimensions. However, this can be handled as in \cite{GS}.

By Plancherel's theorem,
\begin{equation*}
\Qa(f,g)=\int \psi(\lambda^{-1}|\xi|)\frac{\widehat{L^* f\delta_{\Gamma_1}}(\xi)\,\overline{\widehat{g\delta_{\Gamma_2}}(\xi)}}{|\xi|^2-\lambda^2},\quad \Qs(f,g)=\int \psi(\lambda^{-1}|\xi|)\frac{\widehat{f\delta_{\Gamma_1}}(\xi)\overline{\widehat{g\delta_{\Gamma_2}}}(\xi)}{|\xi|^2-\lambda^2}.
\end{equation*}


\noindent Thus, to prove the lemma, we only need estimate 
\begin{equation}
\label{eqn:restrictedDualPlancherel}
\int \psi(\lambda^{-1}|\xi|)\frac{F(\xi)\,G(\xi)}{|\xi|^2-\lambda^2}
\end{equation}
where by \eqref{eqn:fourierRestrictionEstimate1}, \eqref{eqn:fourierRestrictionEstimate2}, \eqref{eqn:fourierRestrictionEstimate3}, and \eqref{eqn:fourierRestrictionEstimate4}
$$\|F\|_{L^2(S_r^{d-1})}+\|\nabla_{\xi}F\|_{L^2(S_r^{d-1})}\leq C\la r\ra^{\delta_1}\|f\|_{L^2(\Gamma)},\quad \|G\|_{L^2(S_r^{d-1})}+\|\nabla_{\xi}G\|_{L^2(S_r^{d-1})}\leq C\la r\ra ^{\delta_2}\|g\|_{L^2(\Gamma)}.$$

Consider first the integral in \eqref{eqn:restrictedDualPlancherel} over $\bigl||\xi|-|\lambda|\bigr|\ge 1$. Since $\bigl||\xi|^2-\lambda^2\bigr|\ge \bigl||\xi|^2-|\lambda|^2\bigr|$, by the Schwartz inequality, \eqref{eqn:fourierRestrictionEstimate1}, and \eqref{eqn:fourierRestrictionEstimate2}
this piece of the integral is bounded by
\begin{align}
\int_{\left||\xi|-|\lambda|\right|\geq 1} \left|\psi(\lambda^{-1}|\xi|)\frac{F(\xi)\,G(\xi)}{|\xi|^2-\lambda^2}\right|&\leq \int_{M\lambda \geq \left|r-|\lambda|\right|\geq 1}\frac{1}{r^2-|\lambda|^2}\int_{S_r^{d-1}}F(r\theta)\,G(r\theta)dS(\theta) dr\nonumber\\
&\leq C\|f\|_{L^2(\Gamma)}\|g\|_{L^2(\Gamma)}
\int_{M|\lambda|\ge |r-|\lambda||\ge 1}
\la r\ra^{\delta_1+\delta_1}\,\bigl|\,r^2-|\lambda|^2\,\bigr|^{-1}dr\nonumber\\
&\leq C\|f\|_{L^2(\Gamma)}\|g\|_{L^2(\Gamma)}\lambda^{\delta_1+\delta_2-1}\int_{M|\lambda|\geq \left|r-|\lambda|\right|\geq 1}\left|r-|\lambda|\right|^{-1}dr\nonumber\\
&\leq C\,|\lambda|^{\delta_1+\delta_2-1}\log |\lambda|\,\|f\|_{L^2(\Gamma)}\|g\|_{L^2(\Gamma)}.\label{eqn:logLoss}
\end{align}

\noindent{\bf Remark:} The estimate \eqref{eqn:logLoss} is the only term where the $\log$ appears. 

Next, if $\Im\lambda\ge 1$, then $\bigl||\xi|^2-\lambda^2\bigr|\ge |\lambda|$, and by 
\eqref{eqn:fourierRestrictionEstimate1}, \eqref{eqn:fourierRestrictionEstimate2} 
$$
\Biggl|\;\int_{||\xi|-|\lambda||\le 1}
\frac{F(\xi)\,G(\xi)}{|\xi|^2-\lambda^2}\,d\xi\;\Biggr|\le C\,|\lambda|^{\delta_1+\delta_2-1}\,\|f\|_{L^2(\Gamma)}\|g\|_{L^2(\Gamma)}.
$$
Thus, we may restrict our attention to $0\le\Im \lambda\leq 1$ and $\bigl||\xi|-|\lambda|\bigr|\le 1$. 

We consider $\Re\lambda\ge 0$, the other case following similarly, and write
$$
\frac{1}{|\xi|^2-\lambda^2}=\frac{1}{|\xi|+\lambda}\;\frac{\xi}{|\xi|}\cdot\nabla_\xi\log(|\xi|-\lambda)\,,
$$
where the logarithm is well defined since $\Im(|\xi|-\lambda)<0$. Let $\chi(r)=1$ for $|r|\le 1$ and vanish for $|r|\ge \frac 32$. We then use integration by parts, together with
\eqref{eqn:fourierRestrictionEstimate1}, \eqref{eqn:fourierRestrictionEstimate2}, \eqref{eqn:fourierRestrictionEstimate3}, and \eqref{eqn:fourierRestrictionEstimate4} to bound
\begin{equation*}
\Biggl|\;\int\chi(|\xi|-|\lambda|)\,\frac{1}{|\xi|+\lambda}\,
F(\xi)\,G(\xi)\,\;\frac{\xi}{|\xi|}\cdot\nabla_\xi\log(|\xi|-\lambda)\,d\xi\;\Biggr|
\le C\,|\lambda|^{\delta_1+\delta_2-1}\,\|f\|_{L^2(\Gamma)}\|g\|_{L^2(\Gamma)}.
\end{equation*}
Now, taking $\delta_1=\delta_2=\alpha$ gives \eqref{eqn:lowFreqSingle}, and taking $\delta_1=\alpha$ and $\delta_2=\beta$ gives \eqref{eqn:lowFreqDouble}.
\end{proof}

We now prove the estimates \eqref{eqn:fourierRestrictionEstimate1} and \eqref{eqn:fourierRestrictionEstimate2}. 
\begin{lemma}
Let $\Gamma\Subset \re^d$ be a compact $C^\infty$ embedded hypersurface. For $L=\partial_{\nu}$ on $\Gamma$, estimate \eqref{eqn:fourierRestrictionEstimate1} holds with $\alpha=1$. Estimate \eqref{eqn:fourierRestrictionEstimate2} holds with $\beta=1/4$. Moreover, if $\Gamma$ is curved, then \eqref{eqn:fourierRestrictionEstimate2} holds with $\beta =1/6.$
\end{lemma}
\begin{proof}
Let $A:H^s(\re^d)\to H^{s-1}(\re^d)$. To estimate
$$\int |\widehat{A^*(f\delta_{\Gamma})}(\xi)|^2\delta(|\xi|-r),$$ 
write 
$$
\la \widehat{A^*(f\delta_{\Gamma})}(\xi)\delta(|\xi|-r),\phi(\xi)\ra =\int\int A^*(f(x)\delta_{\Gamma})\delta(|\xi|-r)\overline{\phi(\xi)e^{i\la x,\xi\ra}} dxd\xi =\int_{\Gamma}f AT_r\phi dx$$
where 
\begin{equation}
\label{eqn:T}
T_r\phi=\int \delta(|\xi|-r)\phi(\xi)e^{i\la x,\xi\ra }d\xi.
\end{equation}

For $\chi\in C_0^\infty(\re^d)$, $\chi T\phi$ is a quasimode of the Laplacian with eigenvalue $\lambda=r$ in the sense of \cite{T}. Thus, we can use the restriction bounds for eigenfunctions found in \cite{BGT}, \cite{HTacy}, \cite{T14}, and \cite{T}, to obtain estimates on $T\phi$. 

To prove \eqref{eqn:fourierRestrictionEstimate2}, let $A=I$. Then, by \cite[Theorem 3]{BGT}, \cite[Theorem 1.7]{T}
\begin{equation}
\label{eqn:restrictEig}
\|\chi T_r\phi\|_{L^2(\Gamma)}\leq r^{\frac 14}\|\chi T\phi\|_{L^2(\re^d)},
\end{equation}
and if $\Gamma$ is curved, then by \cite[Theorem 1.3]{HTacy}
\begin{equation}
\label{eqn:restrictEigCurved}
\|\chi T_r\phi\|_{L^2(\Gamma)}\leq r^{\frac 16}\|\chi T\phi\|_{L^2(\re^d)}.
\end{equation}

\noindent{\bf Remark:} Estimates \eqref{eqn:restrictEig} and \eqref{eqn:restrictEigCurved} continue to hold on $C^{1,1}$ and $C^{2,1}$ curved hypersurfaces respectively (\cite{Blair} \cite{GS}).

Next, we take $A=L$ to obtain \eqref{eqn:fourierRestrictionEstimate1}. Observe that 
$$\chi LT_r\phi =L\chi T_r\phi +[\chi,L]T_r\phi$$
with $[\chi,L]\in C_0^\infty(\re^d)$. Therefore, $[\chi,L]T_r\phi$ is a quasimode of the Laplacian with eigenvalue $r$. 

Hence, using the fact that $L=\partial_{\nu}$ on $\Gamma$ together with \cite[Theorem 0.3]{T14}, we can estimate $LT\phi$.
\begin{equation}
\label{eqn:restrictEigNormal}\|\chi LT_r\phi\|_{L^2(\Gamma)}\leq  \|L\chi T_r\phi\|_{L^2(\Gamma)}+\|[L,\chi ]T_r\phi\|_{L^2(\Gamma)}\leq Cr\|\chi T_r\phi \|_{L^2(\re^d)}.
\end{equation}

To complete the proof of the Lemma, we estimate $\|\chi T\phi\|_{L^2(\re^d)}.$ We have that 
$$\|\chi T_r\phi\|_{L^2(\re^d)}=\|\hat{\chi}*g\delta(|\xi|-r)\|_{L^2(\re^d)}.$$ Therefore, 
\begin{align*}
\|\hat{\chi}*g\delta(|\xi|-r)\|_{L^2(\re^d)}^2&=\int \left|\int_{S_r^{d-1}}\hat{\chi}(\xi-\eta)g(\eta)d\eta\right|^2d\xi\\
&\leq \|g\|_{L^2(S_r^{d-1})}^2\int \int_{S_r^{d-1}} |\hat{\chi}(\xi-\eta)|^2d\eta d\xi\\
&\leq \|g\|_{L^2(S_r^{d-1})}^2 \int \int _{S_r^{d-1}}C_N\la |\xi|-r\ra^{-N}d\eta d\xi \leq C \|g\|_{L^2(S_r^{d-1})}^2.
\end{align*}

Combining this with \eqref{eqn:restrictEig}, \eqref{eqn:restrictEigCurved} and \eqref{eqn:restrictEigNormal} completes the proof of the Lemma.
\end{proof}

To complete the proof of Theorem \ref{thm:main} we need an estimate on the high frequency component of $\Sl$ and $\Dl$. 
Let $\gamma^{\pm}:H^s(\Omega^{\pm})\to H^{s-1/2}(\partial\Omega)$, $s>1/2$ denote the restriction map where $\Omega^+=\Omega$ and $\Omega^-=\re^d\setminus\overline{\Omega}$. Then we have
\begin{lemma}
\label{lem:outsideSphere}
Let $M>1$ and $\psi\in C_0^\infty(\re)$ with $\psi\equiv 1$ for $|\xi|<M$. Suppose that $\partial\Omega$ is a finite union of compact embedded $C^\infty$ hypersurfaces. Then
\begin{align}
\gamma R_\lambda^+(1-\psi(\lambda^{-1}|D|))\gamma^*&=O_{L^2(\partial\Omega)\to L^2(\partial\Omega)}(\la \lambda\ra^{-1}),\label{eqn:estHighSingle}\\
\gamma^{\pm} R_\lambda^+(1-\psi(\lambda^{-1}|D|))L^*\gamma^*&=O_{L^2(\partial\Omega)\to L^2(\partial\Omega)}(1).\label{eqn:estHighDouble}
\end{align}
\end{lemma}
\begin{proof}
Let $h^{-1}=\lambda$. Then $R_{h^{-1}}^+(1-\psi(hD))\in h^2\Psi^{-2}$ (see, for example, \cite[Theorem 3]{Galk}) where $\Psi^k$ denotes the class of semiclassical pseudodifferential operators of order $k$ (see \cite{Z} for a detailed account of the theory of semiclassical analysis).

Now, let $\Gamma_1,\, \Gamma_2\Subset \re^d$ be embedded $C^\infty$ hypersurfaces and denote by $\gamma_i:H^s(\re^d)\to H^{s-1/2}(\Gamma_i)$, $s>1/2$ and by $\gamma_i^*$ its adjoint. Then 
\begin{equation}\label{eqn:restrictBound}\gamma_i=O_{H^s(\re^d)\to H^{s-1/2}(\Gamma_i)}(h^{-1/2}).
\end{equation}
Hence, we have 
$$\gamma_iR^+_{h^{-1}}(1-\psi(|hD|))\gamma_j^*=O_{L^2\to L^2}(h).$$
Since $\gamma=\sum_i\gamma_i$, we have proven estimate \eqref{eqn:estHighSingle}

The strategy for obtaining the bound \eqref{eqn:estHighDouble} will be to compare $\Dl$ at high frequency with the corresponding operators for $\lambda=0$.  First, observe that $R_0^+(1-\psi(|hD|)\in h^2\Psi^{-2}$.

We consider 
$$A_{h^{-1}}:=(R_{h^{-1}}^+-R_0^+)(1-\psi(|hD|))=h^{-2}R_{h^{-1}}^+R_0^+(1-\psi(|hD|)).$$ 
Hence, $A_{h^{-1}}\in h^2\Psi^{-4}$. We will bound 
$$B_h:=\gamma_iA_{h^{-1}}(1-\psi(|hD|))L^*\gamma_j^*.$$
To do so, consider the adjoint
$$B_h^*=\gamma_jL(1-\psi(|hD|))A_{h^{-1}}\gamma_i^*.$$
Then, observe that since $\Gamma_j$ is smooth, we may extend $L$ off of $\Gamma_j$ to a smooth vector field, $\tilde{L}$, without changing $B_h^*.$ Hence, using \eqref{eqn:restrictBound}, and the fact that $\tilde{L}=O_{H^s\to H^{s-1}}(h^{-1})$, we have that $B_h=O_{L^2\to L^2}(1).$

Now, by \cite[Section 5]{SMT}
$$\gamma^{\pm}R_0^+(1-\psi(|hD|))L^*\gamma: L^2(\partial\Omega)\to L^2(\partial\Omega)$$
for $\partial\Omega$ a finite union of compact embedded $C^\infty$ hypersurfaces. Hence, 
$$\gamma^{\pm}R_{h^{-1}}^+(1-\psi(|hD|))L^*\gamma=\gamma^{\pm}R_0^+(1-\psi(|hD|))L^*\gamma +\gamma A_{h^{-1}}L^*\gamma^*=O_{L^2(\partial\Omega)\to L^2(\partial\Omega)}(1)$$
and we have proven \eqref{eqn:estHighDouble}
\end{proof}

\noindent Taking $\partial\Omega=\bigcup_{i}\Gamma_i$ and applying Lemmas \ref{lem:Q} and Lemma \ref{lem:outsideSphere} finishes the proof of Theorem \ref{thm:main}.

\subsection{Sharpness of the single layer operator estimates}\label{sec:sharpSLO}
We now show that the estimates on $\Sl$ in Theorem \ref{thm:main} are sharp modulo the log losses. We use a different approach from that in \cite[Theorem 4.2]{C-WGLL} where the authors construct examples giving the same lower bounds. These examples rely on the concentration of semiclassical singularities in small neighborhoods of glancing rays.

First, recall as above that the spectral measure operator is denoted by $dE_\lambda$ (see \eqref{eq:Stone}).
Then,
\begin{equation}
\label{eqn:specSingleLayer}\gamma dE_{\lambda}\gamma^*=\frac{\Sl-\mathcal{S}_{\lambda}^-}{2\pi i}.
\end{equation}

\noindent But, $ dE_{\lambda}$ has kernel 
$$\frac{1}{2\lambda}(2\pi)^{-d}\int_{S^{d-1}_\lambda}e^{i \la x-y,\xi\ra}d\xi. $$  

Thus, 
$$\gamma dE_{\lambda}\gamma^* = C_d\lambda^{-1}\gamma T_{\lambda} T_{\lambda}^*\gamma^*$$ where $T_\lambda$ is the operator in \eqref{eqn:T}. By \cite{BGT}, \cite{HTacy}, the estimates \eqref{eqn:restrictEig} and \eqref{eqn:restrictEigCurved} are sharp and hence for $\lambda\geq \lambda_0$,
$$\|\gamma dE_{\lambda} \gamma^*\|_{L^2(\Gamma)\to L^2(\Gamma)}\geq \begin{cases}C\lambda^{-\frac 12}&\Gamma \text{ general}\\
C \lambda^{-\frac 23} &\Gamma \text{ curved}\end{cases}.$$

Putting this together with \eqref{eqn:specSingleLayer} gives that 
$$\|\Sl\|_{L^2(\Gamma)\to L^2(\Gamma)}\geq \begin{cases}C\lambda^{-\frac 12}&\Gamma \text{ general}\\
C \lambda^{-\frac 23} &\Gamma \text{ curved}\end{cases}$$
as desired.

\subsection{Sharpness of the double layer operator estimates}\label{sec:sharpDLO}

We will show that there exist smooth embedded hypersurfaces $\Gamma$ such that for $\lambda\geq \lambda_0$,
\begin{equation}
\label{eqn:sharpDLO}
\|\Dl\|_{L^2(\Gamma)\to L^2(\Gamma)}\geq\begin{cases}C \lambda^{1/4}&\Gamma \text{ general}\\
C \lambda^{1/6}&\Gamma \text{ curved}\end{cases}.
\end{equation}
In the flat case, our examples are adaptations of those given in \cite[Theorems 4.6]{C-WGLL} to higher dimensions. However, in the curved case, the example we provide is more subtle and improves the lower bound $\|\Dl\|_{L^2(\Gamma)}\geq C\lambda^{1/8}$ given in \cite[Theorem 4.7]{C-WGLL} 

The idea will be to use a family functions which is microlocalized at a point $((x',0),\xi')\in T^*\Gamma$ such that  $|\xi'|<1$ and the geodesic
$$\{(x',0)+t(\xi',\sqrt{1-|\xi'|^2})\,:\,t\in \re\}$$ 
is tangent to $\Gamma$ at some point away from $(x',0).$

\subsubsection{Flat case}

Let 
$$\Gamma_1:=\{(x_1,x_2,x')\in \re^d\,:\,1/2 <x_1<3/2\,, x_2=0\,,|x'|<1\}$$
$$ \Gamma_2:=\{(x_1,x_2,x')\in \re^d\,:\,x_1=0\,, x_2^2+|x'|^2<1\}.$$
Let $\chi \in C_c^\infty (\re^{d-1})$ have $\chi \geq 0$, $\|\chi\|_{L^2}=1$, and $\hat{\chi}(0)\geq 1/2$. That is
$$\int \chi(x_2,x')dx_2dx'\geq 1/2.$$
Then, denote by $\chi_\lambda:=\chi(M\lambda^{\gamma}(x_2,x'))$ and observe that 
$$\|\chi_{\lambda}\|_{L^2}= C_M\lambda^{-(d-1)\gamma/2},\quad \int\chi_{\lambda}dx_2dx'\geq C_M \lambda^{-(d-1)\gamma}$$
where $M>0$ will be chosen later and $\gamma\geq1/2.$

Now, let $\Gamma\Subset\re^d$ be a smooth embedded hypersurface such that $\Gamma_1\cup\Gamma_2\subset \Gamma$. Suppose also that $f\in L^2(\Gamma)$ is supported on $\Gamma_2.$ Then,
\begin{equation*}
\Dl f|_{\Gamma_1}=\int_{\Gamma_2}(\partial_{\nu_y}R_\lambda^+(x-y))f(y)dy
\end{equation*}

Now, for $|x-y|>\epsilon$, 
\begin{equation} 
\label{eqn:asympDl}
\partial_{\nu_y}R_{\lambda}^+(x-y)=C_d\lambda^{d-1}\frac{\la x-y,\nu_y\ra}{|x-y|}e^{i\lambda|x-y|}\left(\lambda^{-(d-1)/2}|x-y|^{(d-1)/2}+O((\lambda|x-y|)^{-(d+1)/2})\right).
\end{equation}
We will consider $\chi_\lambda$ as a function in $L^2(\Gamma_2)$. Thus, since for $x\in \Gamma_1$ and $y\in \Gamma_2$, $|x-y|\geq \e$, we consider 
$$\lambda^{(d-1)/2}\int_{\Gamma_2}\frac{e^{i\lambda|x-y|}\la x-y,\nu_y\ra}{|x-y|^{(d+1)/2}}\chi_\lambda(y)dy.$$
We are interested in obtaining lower bounds for the $L^2$ norm on $\Gamma_1$. In particular, let $\psi\in C_0^\infty(\re)$ with $\psi(z)=1$ for $|z|\leq 1$. Then, let $\psi_{\lambda,1}(z)=\psi(M\lambda^{\gamma}|z|)$ and $\psi_{\lambda,2}(z)=\psi(M\lambda^{\gamma_2}|z|).$ 

We estimate
$$u=\psi_{\lambda,2}(x_1-1)\psi_{\lambda,1}(x')\lambda^{(d-1)/2}\int_{\Gamma_2}\frac{e^{i\lambda|x-y|}\la x-y,\nu_y\ra}{|x-y|^{(d+1)/2}}\chi_\lambda(y)dy$$
on $\Gamma_1.$
For  $x\in \Gamma_1\cap \supp \psi_{\lambda,1}(x')\psi_{\lambda,2}(x_1-1)$ and $y\in \supp \chi_\lambda$
\begin{equation}
\label{eqn:normalEst}
\frac{\la x-y,\nu_y\ra}{|x-y|}=1+O(\lambda^{-2\gamma}),\quad |x-y|=x_1(1+O(\lambda^{-2\gamma}))
\end{equation}
Hence, we have 
$$u|_{\Gamma_1}= C_d\psi_{\lambda,2}(x_1-1)\psi_{\lambda,1}(x')\lambda^{(d-1)/2}\frac{e^{i\lambda x_1}}{x_1^{(d-1)/2}}\int (1+O(\la\lambda^{1-2\gamma}\ra M^{-2})+O(\la M^{-2}\lambda ^{-2\gamma}\ra)\chi_\lambda dy$$
and on $|x'|<M^{-1}\lambda^{-\gamma}$,
$$u|_{\Gamma_1}(x_1,x')\geq C\psi_{\lambda,2}(x_1-1)  \lambda ^{(d-1)/2}\int \chi_\lambda dy\geq C \lambda ^{(d-1)/2-(d-1)\gamma }.$$
So, 
$$\|u\|^2_{L^2(\Gamma_1)}\geq C\int_{\Gamma_1\cap |x'|<C\lambda^{-\gamma}}\psi_{\lambda,2}^2(x_1-1)\lambda^{d-1-2(d-1)\gamma}\geq C\lambda^{d-1-(3d-4)\gamma-\gamma_2}.$$
Thus, using elementary estimates on the remainder terms 
$$\|\Dl \chi_\lambda\|\geq C\|u\|\geq C\lambda^{\frac{d-1-(3d-4)\gamma-\gamma_2}{2}}.$$
Hence, 
$$\frac{\|\Dl \chi_\lambda\|}{\|\chi_\lambda\|}\geq C\lambda^{\frac{(d-1)(1-2\gamma)+\gamma-\gamma_2}{2}}.$$
Thus, choosing $\gamma =1/2$, $\gamma_2=0$ and $M$ large enough,
$$\|\Dl\chi_\lambda \|\geq C\lambda^{1/4}\|\chi_\lambda\|$$
as desired.

\subsubsection{Curved case}

In order to obtain the lower bound in the curved case, we will need to arrange to hypersurfaces, $\Gamma_1$ and $\Gamma_2$ parametrized respectively by $\gamma,\,\,\sigma:B(0,\e)\subset \re^{d-1}\to \re^d$ such that 
$$|\gamma(x)-\sigma(y)|=|\gamma(0)-\sigma(0)|+O(|x_{1}-y_{1}|^3)+O(|x'-y'|^2)$$
where $x=(x_1,x')\in \re^{d-1}.$
To do this, let $\tilde{\gamma}:(-\e,\e)\to \re^2$ be a smooth unit speed curve with curvature $\kappa(t)=\|\gamma''(t)\|$ and normal vector $n(t)=\gamma''(t)/\kappa(t)$. We assume $\kappa(0)\neq 0$ and $\kappa'(0)\neq 0.$ Then, let $\tilde{\sigma(t)}$ be the loci of the osculating circle for $\tilde{\gamma}(t)$. 
\begin{figure}[htbp]
\centering
\def\svgwidth{3in}
\input{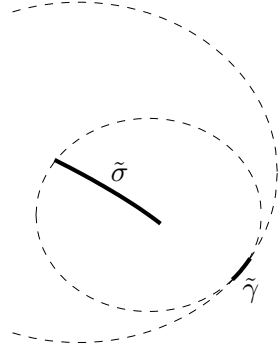}
\caption{We show an example of a curve $\tilde{\gamma}$ and its loci of osculating circles, $\tilde{\sigma}$.}
\label{fig:example}
\end{figure}
That is,
$$\tilde{\sigma}(t)=\tilde{\gamma}(t)+\frac{n(t)}{\kappa(t)}.$$

Finally, define
$$\gamma(x):=(\tilde{\gamma}(x_1)+n(x_1)|x'|^2,x')\,,\quad \quad \sigma(x):=(\tilde{\sigma}(x_1)+\gamma'(x_1)|x'|^2,x').$$
Then we have 
$$|\gamma(y)-\sigma(x)|^2=|\tilde{\gamma}(y_1)-\tilde{\sigma}(x_1)|^2+O(|x'|^2+|y'|^2)+|x'-y'|^2.$$
Let $d(x_1,y_1)=|\tilde{\gamma}(y_1)-\tilde{\sigma}(x_1)|.$
Then, 
$$\partial_{x_1}d(x_1,x_1)=\partial_{x_1}^2d(x_1,x_1)=0.$$
Hence, 
$$d(x_1,y_1)=d(x_1,x_1)+O(|x_1-y_1|^3)$$
and we have near $x=y=0$,
$$|\gamma(y)-\sigma(x)|=|\gamma(0)-\sigma(0)|+O(|x_1-y_1|^3)+O(|x'|^2+|y'|^2).$$
Moreover, 
$$\frac{\la \sigma(x)-\gamma(y),\nu_y\ra}{|\gamma(y)-\sigma(x)|}=1+O(|x-y|).$$
Now, with $\chi\in C_0^\infty(\re^{d-1})$, let $$\chi_\lambda=\chi(M(\lambda^{\gamma_1}x_1,\lambda^{\gamma_2}x')).$$
Then, 
$$\|\chi_\lambda\|_{L^2(\re^{d-1})}=C_M\lambda^{-\frac{d-2}{2}\gamma_2-\frac{1}{2}\gamma_1}\,,\quad \int_{B(0,\e)}\chi_{\lambda}dx_1dx'\geq C_M\lambda^{-(d-2)\gamma_2-\gamma_1}.$$
Next, define $\chi_{\lambda,1}\in L^2(\Gamma_1)$ by $\chi_{\lambda,1}(\gamma(y)):=\chi_\lambda(y)$ and $\chi_{\lambda,2}\in L^2(\Gamma_2)$ by $\chi_{\lambda,2}(\sigma(x)):=\chi_\lambda(x).$

Then 
\begin{equation} \label{eqn:sizeEst}\|\chi_{\lambda,1}\|_{L^2(\Gamma_1)}\,,\, \|\chi_{\lambda,2}\|_{L^2(\Gamma_2)}\geq C\|\chi_{\lambda}\|_{L^2(\re^{d-1})}, \int_{\Gamma_1}\chi_{\lambda,1}\,,\,\int_{\Gamma_2}\chi_{\lambda,2}\geq C\int_{B(0,\e)}\chi_\lambda.\end{equation}
Moreover, for $x\,,\,y\in \supp \chi_{\lambda}$
\begin{equation}
\label{eqn:absEst}
|\gamma(y)-\sigma(x)|=|\gamma(0)-\sigma(0)|+O(M^{-1}(\lambda^{-3\gamma_1}+\lambda^{-2\gamma_2}))\,, \quad \frac{\la \sigma(x)-\gamma(y),\nu_y\ra}{|\gamma(y)-\sigma(x)|}=1+O(\lambda^{-\gamma_1}+\lambda^{-\gamma_2}).
\end{equation}

Hence, choosing $\gamma_1=1/3$ and $\gamma_2=1/2$ and $M$ large enough, using \eqref{eqn:sizeEst}, \eqref{eqn:absEst} in \eqref{eqn:asympDl} we have
$$\|\chi_{\lambda,2}(x)\Dl\chi_{\lambda,1}\|_{L^2(\Gamma_2)}\geq C\lambda^{\frac{d-1}{2}-\frac{d-2}{2}-\frac{1}{3}-\frac{d-2}{4}-\frac{1}{6}}$$
which implies
$$\|\Dl\chi_{\lambda,1}\|_{L^2(\Gamma_2)}\geq C\lambda^{\frac{1}{6}}\|\chi_{\lambda,1}\|_{L^2(\Gamma_1)}.$$

All that remains to show is that $\Gamma_2$ and $\Gamma_1$ can be chosen so that they are curved. To see this, let $\tilde{\gamma}$ be a unit speed reparametrization of $t\mapsto (t+1,(t+1)^2).$ (This example is shown in Figure \ref{fig:example}.) Then, a parametrization of $\Gamma_1$ is given by 
$$(t,x')\mapsto\left((t+1,(t+1)^2)+\frac{(-2(t+1),1)}{\sqrt{1+4(t+1)^2}}|x'|^2,x'\right)$$
and a parametrization of $\Gamma_2$ is given by
$$(t,x')\mapsto\left(\left(-4(t+1)^3,3(t+1)^2+\frac{1}{2}\right)+\frac{(1,2(t+1))}{\sqrt{1+4(t+1)^2}}|x'|^2,x'\right).$$
Then, a simple calculation verifies that near $(0,0)$ these surfaces are curved. Hence, letting $\Gamma$ be a curved hypersurface containing $\Gamma_1$ and $\Gamma_2$ completes the proof of the estimate \eqref{eqn:sharpDLO}

\bibliography{Layer_Ref}

\begin{thebibliography}{10}

\bibitem{BLR}
C.~Bardos, G.~Lebeau, and J.~Rauch.
\newblock Sharp sufficient conditions for the observation, control, and
  stabilization of waves from the boundary.
\newblock {\em SIAM J. Control Optim.}, 30(5):1024--1065, 1992.

\bibitem{BH}
A.~Barnett and A.~Hassell.
\newblock Fast computation of high frequency dirichlet eigenmodes via the
  spectral flow of the interior neumann-to-dirichlet map.
\newblock {\em arXiv:1112.5665}, 2011.

\bibitem{Blair}
M.~Blair.
\newblock $l^q$ bounds on restrictions of spectral clusters to submanifolds for
  low regularity metrics.
\newblock {\em Analysis and PDE}, To appear.

\bibitem{BGT}
N.~Burq, P.~G{\'e}rard, and N.~Tzvetkov.
\newblock Restrictions of the {L}aplace-{B}eltrami eigenfunctions to
  submanifolds.
\newblock {\em Duke Math. J.}, 138(3):445--486, 2007.

\bibitem{C-WGLL}
S.~N. Chandler-Wilde, I.~G. Graham, S.~Langdon, and M.~Lindner.
\newblock Condition number estimates for combined potential boundary integral
  operators in acoustic scattering.
\newblock {\em J. Integral Equations Appl.}, 21(2):229--279, 2009.

\bibitem{C-WGLS}
S.~N. Chandler-Wilde, I.~G. Graham, S.~Langdon, and E.~A. Spence.
\newblock Numerical-asymptotic boundary integral methods in high-frequency
  acoustic scattering.
\newblock {\em Acta Numer.}, 21:89--305, 2012.

\bibitem{CK}
D.~Colton and R.~Kress.
\newblock {\em Inverse acoustic and electromagnetic scattering theory},
  volume~93 of {\em Applied Mathematical Sciences}.
\newblock Springer-Verlag, Berlin, second edition, 1998.

\bibitem{DLMF}
{NIST Digital Library of Mathematical Functions}.
\newblock http://dlmf.nist.gov/, Release 1.0.6 of 2013-05-06.

\bibitem{ET}
L.~El-Hajj and J.~Toth.
\newblock Intersection bounds for nodal sets of planar neumann eigenfunctions
  with interior analytic curves.
\newblock {\em arXiv:1211.3395}, 2012.

\bibitem{FJR}
E.~B. Fabes, M.~Jodeit, Jr., and N.~M. Rivi{\`e}re.
\newblock Potential techniques for boundary value problems on
  {$C^{1}$}-domains.
\newblock {\em Acta Math.}, 141(3-4):165--186, 1978.

\bibitem{FS}
X.~Feng and D.~Sheen.
\newblock An elliptic regularity coefficient estimate for a problem arising
  from a frequency domain treatment of waves.
\newblock {\em Trans. Amer. Math. Soc.}, 346(2):475--487, 1994.

\bibitem{F}
G.~B. Folland.
\newblock {\em Introduction to partial differential equations}.
\newblock Princeton University Press, Princeton, NJ, second edition, 1995.

\bibitem{Galk}
J.~Galkowski.
\newblock Distribution of resonances in lossy scattering.
\newblock {\em arXiv:1404.3709}.

\bibitem{GS}
J.~Galkowski and H.~Smith.
\newblock Restriction bounds for the free resolvent and resonances in lossy
  scattering.
\newblock {\em arXiv:1401.6243}.

\bibitem{HTacy}
A.~Hassell and M.~Tacy.
\newblock Semiclassical {$L^p$} estimates of quasimodes on curved
  hypersurfaces.
\newblock {\em J. Geom. Anal.}, 22(1):74--89, 2012.

\bibitem{HTao02}
A.~Hassell and T.~Tao.
\newblock Upper and lower bounds for normal derivatives of {D}irichlet
  eigenfunctions.
\newblock {\em Math. Res. Lett.}, 9(2-3):289--305, 2002.

\bibitem{HTao10}
A.~Hassell and T.~Tao.
\newblock Erratum for ``{U}pper and lower bounds for normal derivatives of
  {D}irichlet eigenfunctions''.
\newblock {\em Math. Res. Lett.}, 17(4):793--794, 2010.

\bibitem{HZ}
A.~Hassell and S.~Zelditch.
\newblock Quantum ergodicity of boundary values of eigenfunctions.
\newblock {\em Comm. Math. Phys.}, 248(1):119--168, 2004.

\bibitem{SMT}
S.~Hofmann, M.~Mitrea, and M.~Taylor.
\newblock Singular integrals and elliptic boundary problems on regular
  {S}emmes-{K}enig-{T}oro domains.
\newblock {\em Int. Math. Res. Not. IMRN}, (14):2567--2865, 2010.

\bibitem{H83}
L.~H{\"o}rmander.
\newblock {\em The analysis of linear partial differential operators. II.
  Differential operators with constant coefficients}.
\newblock Springer-Verlag, Berlin, 1983.

\bibitem{H85}
L.~H{\"o}rmander.
\newblock {\em The analysis of linear partial differential operators. IV.
  Fourier integral operators}.
\newblock Springer-Verlag, Berlin, 1985.

\bibitem{Hu}
R.~Hu.
\newblock {$L^p$} norm estimates of eigenfunctions restricted to submanifolds.
\newblock {\em Forum Math.}, 21(6):1021--1052, 2009.

\bibitem{JK}
D.~S. Jerison and C.~E. Kenig.
\newblock Boundary value problems on {L}ipschitz domains.
\newblock In {\em Studies in partial differential equations}, volume~23 of {\em
  MAA Stud. Math.}, pages 1--68. Math. Assoc. America, Washington, DC, 1982.

\bibitem{MU}
R.~Melrose and G.~Uhlmann.
\newblock Lagrangian intersection and the {C}auchy problem.
\newblock {\em Comm. Pure Appl. Math.}, 32(4):483--519, 1979.

\bibitem{RT}
D.~Robert and H.~Tamura.
\newblock Semiclassical estimates for resolvents and asymptotics for total
  scattering cross-sections.
\newblock {\em Ann. Inst. H. Poincar\'e Phys. Th\'eor.}, 46(4):415--442, 1987.

\bibitem{Sp}
E.~Spence.
\newblock Wavenumber-explicit bounds in time-harmonic acoustic scattering.
\newblock {\em Available at \url{http://people.bath.ac.uk/eas25/Sp13.pdf}},
  2013.

\bibitem{T}
M.~Tacy.
\newblock Semiclassical {$L^p$} estimates of quasimodes on submanifolds.
\newblock {\em Comm. Partial Differential Equations}, 35(8):1538--1562, 2010.

\bibitem{T14}
M.~Tacy.
\newblock Semiclassical {$L^{2}$} estimates for restrictions of the
  quantisation of normal velocity to interior hypersurfaces.
\newblock {\em arXiv:1403.6575}, 2014.

\bibitem{Ta}
D.~Tataru.
\newblock On the regularity of boundary traces for the wave equation.
\newblock {\em Ann. Scuola Norm. Sup. Pisa Cl. Sci. (4)}, 26(1):185--206, 1998.

\bibitem{Taylor}
M.~Taylor.
\newblock {\em Partial differential equations II. Qualitative studies of linear
  equations. Second edition}, volume 116 of {\em Applied Mathematical
  Sciences}.
\newblock Springer, New York, 2011.

\bibitem{TZ09}
J.~A. Toth and S.~Zelditch.
\newblock Counting nodal lines which touch the boundary of an analytic domain.
\newblock {\em J. Differential Geom.}, 81(3):649--686, 2009.

\bibitem{TZ12}
J.~A. Toth and S.~Zelditch.
\newblock Quantum ergodic restriction theorems. {I}: {I}nterior hypersurfaces
  in domains wth ergodic billiards.
\newblock {\em Ann. Henri Poincar\'e}, 13(4):599--670, 2012.

\bibitem{TZ13}
J.~A. Toth and S.~Zelditch.
\newblock Quantum ergodic restriction theorems: manifolds without boundary.
\newblock {\em Geom. Funct. Anal.}, 23(2):715--775, 2013.

\bibitem{VZ}
A.~Vasy and M.~Zworski.
\newblock Semiclassical estimates in asymptotically {E}uclidean scattering.
\newblock {\em Comm. Math. Phys.}, 212(1):205--217, 2000.

\bibitem{V}
G.~Verchota.
\newblock Layer potentials and regularity for the {D}irichlet problem for
  {L}aplace's equation in {L}ipschitz domains.
\newblock {\em J. Funct. Anal.}, 59(3):572--611, 1984.

\bibitem{Z}
M.~Zworski.
\newblock {\em Semiclassical analysis}, volume 138 of {\em Graduate Studies in
  Mathematics}.
\newblock American Mathematical Society, Providence, RI, 2012.

\end{thebibliography}
\bibliographystyle{abbrv} 
\end{document}